\documentclass[reqno]{amsart}
\usepackage{a4wide}

\title[Degenerate fast-decay mobility ]{Nonlinear diffusion equations with degenerate\\
 fast-decay mobility by coordinate transformation
}

 \author[N. Ansini and S. Fagioli]{}
 
 \email{ansini@mat.uniroma1.it}
\email{simone.fagioli@univaq.it}
 
\newtheorem{defin}{Definition}[section]
\newtheorem{thm}{Theorem}[section]
\newtheorem{prop}{Proposition}[section]
\newtheorem{lem}{Lemma}[section]

\newtheorem{rem}{Remark}[section]

\newcommand{\R}{\mathbb{R}}
\newcommand{\Rd}{{\mathbb{R}^{d}}}

\newcommand{\dive}{\mathrm{div}}
\newcommand{\mP}{{\mathcal{P}}}
\newcommand{\mF}{{\mathcal{F}}}
\newcommand{\setR}{\mathbb{R}}
\newcommand{\dd}{\,\mathrm{d}}

\newcommand{\ep}{\epsilon}
\newcommand{\dx}{\,dx}
\newcommand{\dy}{\,dy}
\newcommand{\dz}{\,dz}

\newcommand{\ie}{; {\it i.e.}, }

\begin{document}


\subjclass[2010]{}

\maketitle

\centerline{\scshape Nadia Ansini}
\medskip
{\footnotesize
        \centerline{Dipartment of Mathematics, Sapienza University of Roma }
    \centerline{ P.le Aldo Moro 2, 00185 Rome,  Italy}
}

\medskip

\centerline{\scshape Simone Fagioli}
\medskip
{\footnotesize
       \centerline{DISIM - Department of Information Engineering, Computer Science and Mathematics, University of L'Aquila}
    \centerline{Via Vetoio 1 (Coppito), 67100 L'Aquila (AQ), Italy}
}

\begin{abstract}
We prove an existence and uniqueness result for solutions to nonlinear diffusion equations with degenerate mobility posed on a bounded interval  for a certain density $u$. In case of \emph{fast-decay} mobilities, namely mobilities functions under a Osgood integrability condition, a suitable coordinate transformation is introduced and a new nonlinear diffusion equation with linear mobility is obtained.
We observe that the coordinate transformation induces a mass-preserving scaling on the density and the nonlinearity, described by the original nonlinear mobility, is included in the diffusive process.
We show that the rescaled density $\rho$ is the unique weak solution to the nonlinear diffusion equation with linear mobility. Moreover, the results obtained for the density $\rho$ allow us to motivate the aforementioned change of variable and to state the results in terms of the original density $u$ without prescribing any boundary conditions.

\end{abstract}

\section{Introduction}
\date{20 gennaio 2019}

Spreading behaviours appear in a large class of phenomena in biology such as animal swarming, chemiotaxis and bacterial movements, but also in modelling pedestrian movements and opinon formation, and it is often in competition with other effects, such as transport driven by external forces (local potentials) and/or aggregation or repulsion induced by the presence of non-local potentials. In order to handle the aforementioned dynamics mathematical models composed by nonlinear aggregation/diffusion/transport equations  were introduced \cite{capasso,okubo,keller_segel,mogilner,OkLe,painter_hillen} and deeply studied in recent years adopting different techniques and investigating possible modeling extensions (see e.g.\cite{bel,blanchet, BDF,CPSW,FR,varadhan2,MatSol} and references therein). The presence of a nonlinear mobility term in the equation may help to improve the ability of the models to catch more sophisticated phenomena. The general form of the equation we are considering is
\begin{equation}\label{eq:general}
\partial_t u = \dive\left(G(x,u)\nabla\left(\Phi(u)+W(x)\right)\right),
\end{equation}
where $u$ is the density population, the function $\Phi$ models the spreading effects and, in general, it is a nonlinear function of the density, $W$ is an external potential.
 Non-linear mobilities functions $G$, depending only on the density $u$ and  degenerating for a certain value $u_{max}>0$, are used to prevent the overcrowding effect, for istance in classical chemotaxis models that may produce blow-up in finite time, see \cite{bel,blanchet,JL,TW}. The presence of such a mobility induce a more $realistic$ behavior in which this phenomenon is prevented: aggregation stops once $u_{max}$ is reached, see \cite{BurDiFDol,bruna_chapman,painter_hillen}.

In this paper we deal with a mobility function of the form, $G(x,u)=g(x)^2 u$, that is linear in $u$ and \emph{non homogeneous} in $x$. Such mobility may model the possible presence of spatial heterogeneity in the domain of $u$. In the sequel we call mobility the function $g(x)$\ie the $x$-dependent part of $G$. We reduce to the one-dimensional initial value problem for nonlinear convection-diffusion equation on bounded intervals  with degenerate mobility 
by considering the following equation
\begin{equation}\label{eq:main_u_1}
  \partial_t u = (g(x)^2  u (\varphi'(u)+W(x))_x)_x\,,
\end{equation}
where $u= u(x,t)$ is defined on the domain $Q_\Omega := \{(x,t)\in \Omega\times [0,+\infty)\}$ with $\Omega=(-1,1)$.
We assume that the mobility function $g:\Omega\rightarrow [0,+\infty)$ (or \emph{inverse metric coefficient}) vanishes at the edges $x=\pm 1$.
We can consider as reference example $g(x)=(1-x^2)^{p/2}$, $p>0$. 
The function $\varphi:[0,+\infty)\rightarrow \R$ represents a \emph{free energy density}, resulting from local repulsive effects or volume filling mechanisms, and $W:\Omega\rightarrow \R$ is the \emph{external potential}.
Since $g$ vanishing at the edges $x=\pm1$, the problem of posing suitable zero-flux boundary conditions arises in order to have a (unique) solution $u$ with constant mass. Roughly speaking,  if we consider  \eqref{eq:main_u_1} as the continuum limit equation of a many particles system and we assume that $g$ vanishes very fast at $x=\pm 1$ then the particles slow down so fast at the boundary that no boundary condition has to be prescribed in order to preserve the total mass of $u$. On the other hand, if $g$ goes to zero very slowly at $x=\pm 1$, a zero-flux boundary condition could tackle the loss of mass.

The formulation of equation \eqref{eq:main_u_1} as gradient flows, in the sense of \cite{AGS}, on a  \emph{modified} Wasserstein space  was first proven in \cite{Lisini} for a class of mobility functions $G:\R^n \to \R^n$ statisfying a uniform ellipticity assumption, 
$$\lambda |\zeta|^2 \leq \langle G(x)\zeta,\zeta\rangle\leq \Lambda |\zeta|^2 $$
 for all $x,\zeta\in\R^n$ and for some $\lambda,\Lambda>0$, inducing a metric coefficient $M=G^{-1}$ that satisfy a similar condition, see also \cite{CGM,CCT,SW}. Unfortunately this result does not apply to our case. 
Therefore, a new mathematical approach is needed in order to prove existence and uniqueness of solutions to  equation \eqref{eq:main_u_1}. Moreover, the mobility degenerating at the boundary models are of high interest  also for applications (see e.g. the modeling of the opinion formation phenomena \cite{Tos} ).

Our approach consists in introducing a suitable coordinate transformation with the aim of getting a Fokker-Planck type equation in a new variable $\rho$ defined on the whole space $\R$ and with homogeneous mobility.  
Indeed, we set $\alpha:\Omega\rightarrow \R$ as
\begin{equation}\label{eq:alpha}
  \alpha(x):=\int_0^x \frac{1}{g(z)}dz.
\end{equation}
By definition of $g$ we have that $\alpha$ is a $C^1(\Omega)$, strictly increasing function. We assume that $g$ satisfies also the Osgood condition
\begin{equation}\label{M1}
\displaystyle{\int_0^1 \frac{1}{g(z)} dz =+\infty}\,,
\end{equation}
that is, the mobility has a \emph{fast-decay} behaviour. The function $\alpha$ is a $1:1$ map from $\Omega$ onto $\R$. Our reference example  $g(x)=(1-x^2)^{p/2}$ is a fast-decay mobility provided $p\geq 2$. Setting the coordinate transformation
$$
y=\alpha(x) \in \R,\quad \forall x\in \Omega,
$$
and the mass preserving  scaling as follows
\begin{equation}\label{eq:scaling}
  u(x,t)=\alpha^\prime(x) \rho(\alpha(x),t)\,,
\end{equation}
we have that, by assumption (\ref{M1}), $\rho$ is defined on \[Q_\R=\{(y,t)\in \R\times [0,+\infty)\}.\]
Replacing the ansatz  \eqref{eq:scaling} into \eqref{eq:main_u_1} we obtain
\begin{equation}\label{eq:main_rho_1}
  \partial_t \rho = (\rho(\varphi'(a(y)\rho) + V)_y)_y\,,
\end{equation}
where
$$
  a(y):= \frac{1}{g(\alpha^{-1}(y))}\,,\qquad V(y):=W(\alpha^{-1}(y))\,.
$$
Therefore, we may conclude that, at least formally,  if $u$ solves \eqref{eq:main_u_1} then $\rho$ solves \eqref{eq:main_rho_1} and vice versa.

There are two main advantages in studying problem \eqref{eq:main_rho_1} in place of \eqref{eq:main_u_1}. First of all, as already observed, the new equation is posed on the whole $\R$, and  no boundary conditions should be prescribed. Moreover, the mobility in the continuity equation is linear and no longer depending on the space variable. 

If $g$ does not satisfy 
(\ref{M1})\ie there exists $l>0$ such that
\begin{equation}\label{M2}
 \displaystyle{\int_0^1 \frac{1}{g(z)} dz =l<+\infty}\,,
\end{equation}
the map $\alpha$ is now a bi-jection from $(-1,1)$ into $(-l,l)$ as e.g. in case of $0\leq p < 2$ for $g(x)=(1-x^2)^{p/2}$. Condition \eqref{M2} corresponds then to \emph{slow-decay} behaviour of the  mobility. We state that the scaling \eqref{eq:scaling} can be still applied, and a new density $\rho(y,t)$ still solves \eqref{eq:main_rho_1}. However, $\rho$ is defined on the bounded spacial domain $(-l,l)$, and a zero-flux boundary condition must be prescribed in order to preserve its total mass.  In a forthcoming paper we explore in details the case of \emph{slow-decay mobility}. 

\bigskip

In the fast-decay mobility case we recognize equation \eqref{eq:main_rho_1} as the formal Wasserstein gradient flow associated to the functional
\begin{equation}\label{functional}
\mF^a[\rho]=\int_{\R}\frac{\varphi(a(y)\rho(y))}{a(y)}dy+\int_\R V(y)\rho(y) dy,
\end{equation}
(see e.g. \cite{AGS}) .
We will recall the basic notions of Wasserstein gradient flow theory in Section \ref{subsec:AGS}.
It is well known by the theory developed in \cite{AGS,otto,S,V2} that \eqref{eq:main_rho_1} has a unique solution in the space of probability measures with finite second moment provided the functional $\mF$ above is \emph{displacement $\lambda$-convex} (in addition to some further technical assumptions)\ie geodesically convex on the Wasserstein space up to a quadratic perturbation. Hence, following the approach as in \cite{DFMatthes}, we will collect conditions on $g$, $\varphi$, and $W$ such that the corresponding functional $\mF$ obtained after the scaling \eqref{eq:scaling} is geodesically $\lambda$-convex. Moreover, we state the existence and uniqueness result for \eqref{eq:main_rho_1} by using minimizing movements method and the by-now classical JKO approach \cite{JKO}, and we reformulate the result for the density $u=u(x,t)$ via the scaling \eqref{eq:scaling}. In particular, we determine the class of initial conditions for $u$ such that a unique solution for \eqref{eq:main_u_1} exists without imposing any boundary condition.

The paper is organized as follows. In Section \ref{Prel} we first derive \eqref{eq:main_rho_1}  using the coordinate transformation and the scaling \eqref{eq:scaling}, then we list the  assumptions and we collect some useful tools and results that we will apply to prove the main result stated in Theorem \ref{main_th}. Section \ref{MainProblem} is devoted to the proof of our main result; more precisely, we prove existence and uniqueness for the rescaled density $\rho$ in Section \ref{sec:JKO} and Section \ref{k-flow}, respectively.  In Section  \ref{E-U} we reformulate the result obtained for $\rho$ in terms of the density function $u$. Finally, in Section \ref{SpecialCases} we focus on three relevant cases as the Heat equation, linear Fokker-Planck equation and Porous Medium equation with degenerate mobility.

\section{Preliminaries}\label{Prel}
In this section we collect general assumptions and properties on functions $a$, $g$, $V$ and $W$ that are involved in the definition of the equations  
    \eqref{eq:main_u_1} and  \eqref{eq:main_rho_1}. Moreover, we derive equation \eqref{eq:main_rho_1} and we recall the notion of Wasserstein gradient flow and the extension version of the Aubin-Lions Lemma.  Finally, we prove the Flow interchange Lemma \ref{Flow_interchange}.

We use the usual notations $h^\prime(z)$ and $\partial_z h$ to denote the first derivative of a function $h$ depending only on one variable and the first order partial derivative for $h$ depending on two variables; respectively. 
To the aim to not overburden the notations, we will use  also any of the following notations $h_z\,, [h]_z\,, (h)_z$ to denote the first derivative or first order partial derivative. We leave the interpretation up to the reader, it will be clear anyway from the context.
Similarly, for the second derivative and for the second order partial derivative. 

\subsection{Derivation of nonlinear convection-diffusion equation on $\R$ with homogeneous mobility}\label{sec:scaling}
We want to derive equation \eqref{eq:main_rho_1} from equation \eqref{eq:main_u_1} by applying the scaling \eqref{eq:scaling}. More precisely, we replace
$$
 u(x,t)=\alpha^\prime(x) \rho(\alpha(x),t),
$$
into \eqref{eq:main_u_1} and we obtain
\begin{align} \label{eq:step0}
  \alpha^\prime\rho_t\circ\alpha = \big(g^2\alpha^\prime\rho\circ\alpha[\varphi'(\alpha^\prime\rho\circ\alpha)+W]_x\big)_x\,.
\end{align}
We define now the functions $a:\R\to\setR_+$ and $V:\R\to\setR$ as 
\begin{equation}\label{eq:new_coefficient}
  a(y):= \frac{1}{g(\alpha^{-1}(y))}\,,\qquad V(y):=W(\alpha^{-1}(y))\,.
\end{equation}
Hence, by \eqref{eq:alpha} we have that
\begin{align}\label{defalpha1}
  a\circ\alpha (x)= \frac{1}{g(x)}=\alpha^\prime(x)\,, \quad V\circ\alpha(x) = W(x)\,,
\end{align}
and
\begin{equation}\label{defalpha2}
\alpha^{\prime\prime}(x)= (a^\prime\circ\alpha) \alpha^\prime\,, \quad W^\prime(x)= \alpha^\prime(x) V^\prime\circ\alpha(x)\,.
\end{equation}
Therefore, we have that
\begin{eqnarray}\label{[]x}
\nonumber [\varphi^\prime(\alpha^\prime\rho\circ\alpha)]_x &=&
\varphi^{\prime\prime}(\alpha^\prime\rho\circ\alpha) [\alpha^{\prime\prime} \rho\circ\alpha + (\alpha^{\prime})^2\partial_y\rho\circ\alpha]\\
\nonumber &=& \alpha^{\prime} \varphi^{\prime\prime}(a \rho\circ\alpha) [a^{\prime} \rho\circ\alpha + a\partial_y\rho\circ\alpha]\\
&=&
\alpha^{\prime}  [\varphi'(a \rho)]_y\circ\alpha\,.
\end{eqnarray}
By applying \eqref{defalpha1}, \eqref{defalpha2}, and \eqref{[]x} we have that the metric factor in \eqref{eq:step0} disappears and the equation \eqref{eq:step0}  becomes
\begin{eqnarray}
  \label{eq:step1}
\nonumber   \alpha^\prime\rho_t\circ\alpha &=& \big(\rho\circ\alpha[\varphi^\prime(a\rho)+V]_y\circ\alpha\big)_x\\
&=& \alpha^\prime \big(\rho[\varphi^\prime(a\rho)+V]_y\big)_y \circ\alpha\,.
\end{eqnarray}
Therefore, we get equation \eqref{eq:main_rho_1}.


\bigskip

\subsection{Main assumptions and properties}\label{Assumptions}
We assume that the mobility function $g:\Omega\to \left[0,1\right]$ is a $C^2(\Omega)$ function satisfying the following
conditions:
\begin{itemize}
    \item[(g1)] $g$ has a maximum point at $g(0) = 1$ and $g(\pm 1) = 0$,
    \item[(g2)] the Osgood condition (\ref{M1}),
    \item[(g3)] there exists a constant $C_g>0$ such that $0\leq (g^\prime)^2-g\, g^{\prime\prime}\leq C_g$. 
\end{itemize}
We collect in the following Proposition some usefull properties of the function $a$ defined in \eqref{eq:new_coefficient}.
\begin{prop}\label{aProperties}
Let $g$ be a function as above satisfying  {\rm (g1), (g2) and (g3)}. Let $\alpha$ and $a$ be defined as in \eqref{eq:alpha} and \eqref{eq:new_coefficient}, respectively. Then, $a:\R\mapsto [1, +\infty)$ is a convex function satisfying the following properties:
\begin{itemize}
\item[(i)] $a(y)\geq a(0)=1$\,;
\item[(ii)] $|a'(y) / a(y)|\le C_a$ and $y\, a'(y) / a(y) \ge 0$\,;
\item[(iii)] the ratio $a''/a$ remains bounded for all $y\in\R$\,.
\end{itemize}
In particular, if $g(x)=(1-x^2)^{p/2}$, with $p\ge 2$, then conditions (i) and (ii) are still satisfied with $C_a=p$. Moreover,
condition (iii) still holds for every $p\ge 2$ and 
 $a^{\prime\prime}$ is bounded for every $p\ge 4$.
\end{prop}
\begin{proof} 
By (g1) and \eqref{defalpha1} we have that $a$ has a global minimum at $y=0$, that implies  condition (i).
By \eqref{eq:alpha} and eqref{defalpha2} we have that
\begin{equation}\label{eq:help2}   
a^\prime\circ\alpha(x) = -\frac{g^\prime(x)}{g(x)}\,,\quad \frac{a^\prime}{a}\circ\alpha(x) = - g^\prime(x)\,;
\end{equation}
hence, we have  that $|a'(y) / a(y)|$ remains bounded. Moreover
\begin{equation} \label{eq:help3}   
a^{\prime\prime}\circ \alpha(x)  = \frac{(g^\prime(x))^2-g^{\prime\prime}(x)g(x)}{g(x)}\,\quad 
\frac{a^{\prime\prime}}{a}\circ \alpha(x)= (g^\prime(x))^2-g^{\prime\prime}(x)g(x)\,;
\end{equation}
therefore by (g3) we get that the function $a$ is convex  and condition (iii) is satisfied. In particular, the convexity of $a$ implies that  $a(y) - a'(y) y \le a(0)$\ie
$$
\frac{a'(y)}{a(y)} y\ge 0\,.
$$ 
In the case $g(x)=(1-x^2)^{p/2}$, a direct computation shows
$$
\frac{a'(y)}{a(y)} = p\,  \alpha^{-1}(y) (1- (\alpha^{-1}(y))^2)^{p/2-1}\,;
$$ 
hence, $|a'(y) / a(y)|\le p$ for every $y\in\R$, and
$$
a^{\prime\prime}(y) = p\,  \Bigl(1+ (\alpha^{-1}(y))^2\Bigr) \Bigl( 1- (\alpha^{-1}(y))^2\Bigr)^{p/2-2}\,.
$$
 For this particular choice of $g$ then  $a^{\prime\prime}$ is bounded for every $p\ge 4$. Moreover, since
$$
\frac{a^{\prime\prime}(y)}{a(y)}=  p\,  \Bigl(1+ (\alpha^{-1}(y))^2\Bigr) \Bigl( 1- (\alpha^{-1}(y))^2\Bigr)^{p-2}
$$
we have that the ratio $a^{\prime\prime}/a$ remains bounded for all $p\ge2$ and $y\in \R$. 
\end{proof}
\bigskip

Let $\varphi:[0,+\infty)\rightarrow \left[0, +\infty\right)$ be a lower semi-continuous and convex function satisfying the following growth conditions:
\begin{itemize}
\item [(D)]  for $m>1$ and $\mu\in \left[m,3m\right)$ there exist two constants $c_m, C_\mu\ge 0$ such that 
$$c_m s^{m-2}\leq \varphi''(s)\leq C_\mu s^{\mu-2},$$
for every $s\ge 0$.
\end{itemize}
Let $W: \Omega\rightarrow\left[0, +\infty\right)$ be a non-negative $C^2(\Omega)$ function.
We further assume that 
\begin{itemize}
    \item[(gW1)] there exists $\lambda\in \R$ such that
$$
\lambda\le g^2  W^{\prime\prime}(x)+g\, g^\prime W^\prime \mbox{ for all }x\in\left[-1,1\right].
$$
\item[(gW2)] there exists $L>0$ such that
$$
\left[g^2(x)W^\prime(x)\right]_x\leq L, \mbox{ for all }x\in\left[-1,1\right].
$$
\end{itemize}

Note that
\begin{align}
  \label{eq:help4}
 V^{\prime\prime}\circ\alpha(x) & = g^2(x)W^{\prime\prime}(x)+  g(x) g^{\prime} (x)W^{\prime} (x)\\
 &= g^2(x)W^{\prime\prime}(x)+\frac12[g^2(x)]_xW^{\prime} (x)\,.
\end{align}

\begin{rem}
We observe that condition (gW1) naturally arises in the porous medium case (see Section \ref{PM}). Indeed, condition (gW1) implies the $\lambda$-convexity of function $V$; while, condition (gW2) implies 
$$
a(y)\left[\frac{V^\prime(y)}{a(y)}\right]_y\leq L,\, \mbox{ for all } y \in \R\,.
$$
\end{rem}
We can now state the main result of the paper.

\begin{thm}\label{main_th}
Let $g:\Omega\to \left[0,1\right]$ be a $C^2(\Omega)$ function under assumptions (g1)-(g3). Let $\varphi:[0,+\infty)\rightarrow \left[0, +\infty\right)$ be a lower semi-continuous and convex function satisfying (D) and be  $W: \Omega\rightarrow\left[0, +\infty\right)$ be a non-negative $C^2(\Omega)$ function under the assumption (gW1)-(gW2). Consider, for $m>1$, the initial condition $u_0\in L^1\cap L^m(\Omega)$ and fix  $T>0$. Then there exists an Holder-continuos curve $u:[0,T]\to L^m(\Omega)$ such that,
\begin{itemize}
\item[(i)] $u\in L^\alpha([0,T]\times \Omega)$ for some $\alpha \in (1,3m)$;
\item[(ii)] $g\partial_x u^{\frac{m}{2}} \in L^2(\left[0,+\infty\right)\times\Omega)$;
\item[(iii)] for almost every $t\in\left[0,+\infty\right)$ and for all $\psi \in C_c^\infty(\Omega)$, we have
\begin{equation}\label{eq:weak_solutions_u}
    \frac{d}{dt}\int_\Omega\psi(x) u(x,t) dx = -\int_\Omega g^2(x)\partial_x\left(  \varphi'(u)+W(x) \right)\partial_x\psi(x)u(x,t)dx.
    \end{equation}
\end{itemize}

\end{thm}
Note that the result above holds without prescribing any boundary condition on $u$.
\smallskip

\subsection{Preliminaries on Wasserstein gradient flows}\label{subsec:AGS}
We recall some basic notions in Optimal Transport theory, see \cite{AGS,S,V2}. Let us denote with $\mP(\R)$ the space of all probability measures on $\R$ and with $\mP_2(\R)$ the set of all probability measures with finite second moment, i.e.
$$
\mP_2(\R)=\left\{\rho\in\mP(\R):m_2(\rho)<+\infty\right\},
$$
where
$$m_2(\rho)=\int_{\Rd}|x|^2\,d\rho(x).$$
Consider now a measure $\rho\in\mP(\R)$ and a Borel map $T:\R^d\to\R^n$. We denote by $T_{\#}\rho$ the push-forward of $\rho$ through $T$, defined by
$$
\int_{\R^n}f(y)\,dT_{\#}\rho(y)=\int_{\R^d}f(T(x))\,d\rho(x) \qquad \mbox{for all $f$ Borel functions on}\ \R^n.
$$
Let us recall the $2$-Wasserstein distance between $\mu_1,\mu_2\in \mP_2(\R)$ defined by
\begin{equation}\label{wass}
W_2^2(\mu_1,\mu_2)=\min_{\gamma\in\Gamma(\mu_1,\mu_2)}\left\{\int_{\R^{2}}|x-y|^2\,d\gamma(x,y)\right\},
\end{equation}
where $\Gamma(\mu_1,\mu_2)$ is the class of all transport plans between $\mu_1$ and $\mu_2$, that is the class of measures $\gamma\in\mP_2(\R)^2$ such that, denoting by $\pi_i$ the projection operator on the $i$-th component of the product space, the marginality condition
$$
(\pi_i)_{\#}\gamma=\mu_i \quad \mbox{for}\ i=1,2,
$$
is satisfied. Setting $\Gamma_0(\mu_1,\mu_2)$ as the class of optimal plans, i.e. minimizers of \eqref{wass}, we can write the Wasserstein distance as
$$
W_2^2(\mu_1,\mu_2)=\int_{\R^2}|x-y|^2\,d\gamma(x,y), \qquad \gamma\in\Gamma_0(\mu_1,\mu_2).
$$


For $I\subset \R$ we consider an absolutely continuous curve in $W_2$, $\rho : I\to \mP_2(\R)$, namely a curve  such that there exists a function $g\in L_{loc}^1(I)$ such that 
\[
 W_2(\rho(t),\rho(s))\leq |\int_s^t g(\tau)d\tau| \quad \mbox{ for all }t,s\in I.
\]

We introduce the concept of \textit{k-flow}, which is linked with the $\lambda$-convexity along geodesics, see \cite{DS,OW} for further details.
\begin{defin}\label{EVI}
A semigroup $G_\Psi :\left[0,+\infty\right]\times\mP_2(\R) \to \mP_2(\R)$ is a $k-$flow for a functional $\Psi :\mP_2(\R) \to \R\cup\left\{+\infty\right\}$ with respect to the Wasserstein distance $W_2$ if, for an arbitrary $\rho \in \mP_2(\R)$, the curve $s\mapsto G_\Psi^s\rho$ is absolutely continuous on  $\left[0,+\infty\right]$ and satisfies the Evolution Variational Inequality (E.V.I.)
\begin{equation}\label{eq:EVI}
 \frac{1}{2}\frac{d^+}{d\sigma}W_2^2( G_\Psi^\sigma\rho,\tilde{\rho})|_{\sigma =s}+\frac{k}{2}W_2^2( G_\Psi^s\rho,\tilde{\rho})\leq \Psi(\tilde{\rho})-\Psi( G_\Psi^\sigma\rho),
\end{equation}
for all $s>0$ and for any $\tilde{\rho}\in \mP_2(\R)$, such that $\Psi(\tilde{\rho})<\infty$.
\end{defin}
\begin{rem}
The symbol $d^+/d\sigma$ stands for the limit superior of the respective difference quotients and equals to the derivative if the latter exists.
\end{rem}
\begin{thm}
Asssume that a functional $\Psi :\mP_2(\R) \to \R\cup\left\{+\infty\right\}$ is $\lambda$-convex (along geodesics), with a modulus of convexity $\lambda\in \R$, that is, along every constant speed geodesic $\rho :\left[0,1\right]\to \mP_2(\R)$
\[ 
 \Psi\left[\rho(t)\right]\leq (1-t)\Psi\left[\rho(0)\right]+t\Psi\left[\rho(1)\right]-\frac{\lambda}{2}t(1-t)W_2^2(\rho(0),\rho(1))
\]  
holds for every $t\in \left[0,1\right]$. Then $\Psi$ posses a uniquely determined $k-$flow, with some $k\geq \lambda$. Conversely, if a functional $\Psi$ posses a $k-$flow, and if is monotonically non-increasing along that flow, then $\Psi$ is $\lambda$-convex, with some $\lambda\geq k$.
\end{thm}

We now recall an extension of the Aubin-Lions Lemma first introduced in \cite{RS}. This result  will be  used later to  prove the existence of weak solutions to \eqref{eq:main_rho_1}. 

\begin{thm}[Extended Aubin-Lions Lemma \cite{RS}]\label{ex_AL}
On a Banach space $X$, let be given
\begin{itemize}
    \item a normal coercive integrand $\mathcal{Y}:X\to [0,\infty]$, i.e. $\mathcal{Y}$ is lower semi-continuous and its sub-levels are relatively compact in $X$;
    \item a pseudo-distance $\mathrm{d}:X\times X \to [0,\infty]$, i.e. $\mathrm{d}$ is lower semi-continuous and $\mathrm{d}(\rho,\eta)=0$ for any $\rho,\eta\in X$ with $\mathcal{Y}(\rho), \mathcal{Y}(\eta)<\infty$ implies $\rho=\eta$.  
\end{itemize}
Let further $U$ be a set of measurable functions $u:[0,T]\to X$, with a fixed $T>0$. If
\begin{equation}\label{eq:al}
    \sup_{u\in U}\int_0^T\mathcal{Y}[u(t))] dt<\infty \mbox{ and }\lim_{h\downarrow 0}\sup_{u\in U}\int_0^{T-h}\mathrm{d}\left(u(t+h),u(t)\right)dt =0, 
\end{equation}
$U$ contains an infinite sequence $\left\{u_n\right\}_{n\in \mathbb{N}}$ that converges in measure (with repect to $t\in [0,T]$) to a limit $u :[0,T] \to X$.
\end{thm}
A key ingredient is the so called \emph{flow interchange lemma}, see \cite{DFMatthes} for further details.
\begin{lem}\label{Flow_interchange}
Let $\Psi:\mP_2(\R)\to\left(-\infty,+\infty\right]$ be a lower semi-continuous functional which posses a $k-$flow $G_\Psi$. Define the dissipation of a functional $\mF$ along $G_\Psi$ by
\[ \mathcal{D}_\Psi\mF^a(\rho):=\limsup_{s\downarrow 0}\frac{1}{s}\left(\mF^a\left[\rho\right]-\mF^a\left[G_\Psi^s\rho\right]\right),\]
for every $\rho\in\mP_2(\R)$. If $\rho_\tau^{n-1}$ and $\rho_\tau^n $ are two consecutive steps in the JKO scheme (\ref{JKO}), then
\begin{equation}\label{Ineq-psi}
 \Psi\left[\rho_\tau^{n-1} \right]-\Psi\left[\rho_\tau^{n} \right]\geq \tau \mathcal{D}_\Psi\mF^a(\rho_\tau^n)+\frac{k}{2}W_2^2(\rho_\tau^n,\rho_\tau^{n-1}).
\end{equation} 
In addition, assume that $G_\Psi$ is such that for every $n\in\mathbb{N}$, the curve $s\mapsto G_\Psi^s\rho_\tau^n$ lies in $L^m(\R)$, it is differentiable for $s>0$ and continuous at $s=0$. Let $\mathcal{R}:\mP_2(\R)\to\left(-\infty,+\infty\right]$ be a functional satisfying
\[ 
 \liminf_{s\downarrow0}\left(-\frac{d}{d\sigma}|_{\sigma=s}\mF^a\left[G_\Psi^\sigma\rho_\tau^n\right]\right)\geq \mathcal{R}\left[\rho_\tau^n\right]\,.
\]
Then the following estimate holds:
 for every $n\in\mathbb{N}$,
\begin{equation}\label{FI_1}
 \Psi\left[\rho_\tau^{n-1} \right]-\Psi\left[\rho_\tau^{n} \right]\geq \tau  \mathcal{R}\left[\rho_\tau^n\right]+\frac{k}{2}W_2^2(\rho_\tau^n,\rho_\tau^{n-1});
\end{equation}
In particular, for every $N\in\mathbb{N}$,
\begin{equation}\label{FI_2}
 \Psi\left[\rho_\tau^N \right]\leq\Psi\left[\rho^0 \right]- \tau\sum_{n=1}^N  \mathcal{R}\left[\rho_\tau^n\right]- \frac{k}{2} \sum_{n=1}^NW_2^2(\rho_\tau^n,\rho_\tau^{n-1}).
\end{equation}

\end{lem} 
\begin{proof}
The proof of (\ref{Ineq-psi}) easily follows as in \cite[Lemma 4.2]{DFMatthes}, after recalling Definition \ref{EVI}, the $W_2$- absolute continuity of the curve $s\mapsto G_\Psi^s\rho_\tau^n$ and the definition of $\rho_\tau^{n}, \rho_\tau^{n-1}$ as in (\ref{JKO}).
\end{proof}

\section{Problems with fast-decay mobilities}\label{MainProblem}
In this section we study existence for solutions to \eqref{eq:main_rho_1} and $\lambda$-convexity property for the related functional $\mF^a$ introduced in \eqref{functional}. Let us recall that the equation \eqref{eq:main_rho_1} obtained in Section \ref{sec:scaling} for the scaled density $\rho$ is
\begin{equation}\label{eq:1}
\rho_t = \left(\rho\left(\varphi'(a\rho)+V\right)_y\right)_y\,\mbox{ for }(t,y)\in \left[0,+\infty\right)\times\R.
\end{equation}
For technical convenience we define the following functions
\begin{equation}\label{eq:F&H}
 F^a(y,\eta)=\frac{1}{a(y)}\varphi(a(y)\eta)\,, \quad H(y,\eta)=\eta F^a(y,\frac{1}{\eta}),
\end{equation}
and we reformulate equation \eqref{eq:1} and the functional $\mF^a$ as 
\begin{equation}\label{eq:2}
 \rho_t = (\rho(F_\eta^a(y,\rho)+V)_y)_y\,.
\end{equation}
and
\begin{equation}\label{mFa}
 \mF^a\left[\rho\right]=\int_\R F^a(a(y),\rho(y))dy+\int_\R V(y)\rho(y)\dy\,.
\end{equation}
respectively. At least formally we can introduce the cumulative distribution function $R$ of $\rho$, defined as
\[
 R(t,y)=\int_{-\infty}^y \rho(t,z)\dz\,,
\]
and its pseudo-inverse function 
\begin{equation}\label{defY}
   Y(t,\omega)=\inf\left\{y : R(t,y)>\omega\right\}, 
\end{equation}
for any $y\in \R$, $w\in \left(0,1\right)$ and $t\geq 0$, respectively. Note that
\begin{align}\label{Yproperty}
  Y_\omega\rho\circ Y = 1\,.
\end{align}
The functions $R$ and $Y$ formally satisfy  the equations
\begin{align}
 &  R_t = R_y \left(F_\eta^a(y,R_y)+V(y)\right)_y,\label{eq:distr}\\
 & Y_t = -\frac{1}{ Y_\omega} \Bigl(F_\eta^a\Bigl(Y,\frac{1}{Y_\omega}\Bigr)+V(Y)\Bigr)_\omega\label{eq:pseudo}\,;
\end{align}
respectively. We will make use of this reformulation in Subsection \ref{k-flow}. In Subsection \ref{sec:JKO} we prove existence for weak solutions to \eqref{eq:1} using the so-called \emph{JKO-scheme}.
\begin{defin}\label{def:weak_solutions}
We say that a curve $\rho:\left[0,T\right]\to \mP_2(\R)$ is a weak solution to \eqref{eq:2} if
\begin{itemize}
    \item[(i)] $\rho\in L^\alpha(\left[0,T\right]\times\R)$, with $\alpha\in (1,\mu)$ for all $T>0$;
    \item[(ii)] $\partial_y \rho^{\frac{m}{2}} \in L^2(\left[0,+\infty\right)\times\R)$;
    \item[(iii)] for almost every $t\in\left[0,+\infty\right)$ and for all $\zeta \in C_c^\infty(\R)$, we have
    \begin{equation}\label{eq:weak_solutions}
    \frac{d}{dt}\int_\R\zeta(y) \rho(y,t) dy = -\int_\R \partial_y\left(  \varphi'(a\rho)+V(y) \right)\partial_y\zeta(y)\rho(y,t)dy .
\end{equation}
\end{itemize}
\end{defin}

\subsection{The minimising movements or JKO scheme}\label{sec:JKO}

We will construct solutions to \eqref{eq:main_rho_1} in the spirit of the so called \emph{minimising movements or implicit-Euler scheme}. The application of such method to the Fokker- Planck equation appears  in \cite{JKO}  and it has been deeply used in the literature in the last years. The scheme consists in constructing recursively a sequence as follows: given $\rho\in \mP_2(\R)$, for a fixed time step $\tau>0$ and for every $\eta\in \mP_2(\R)$ we introduce the penalization functional $\Phi_\tau\left(\rho;\eta\right)$ defined by
\begin{equation}\label{eq:pen_fun}
\Phi_\tau\left(\rho;\eta\right)=\frac{1}{2\tau}W_2^2\left(\rho,\eta\right)+\mF^a\left[\rho\right].
\end{equation}
If $\rho^0\in \mP_2(\R)$ is an initial condition with $\mF^a\left[\rho^0\right]<\infty$ we define the sequence  $\left\{\rho_\tau^n\right\}_{n\in\mathbb{N}}$ inductively as follows
\begin{itemize}
\item[(i)] $\rho_\tau^0=\rho^0$;
\item[(ii)] for $n\geq 1$, given $\rho_\tau^{n-1}$, define $\rho_\tau^n$ as
\begin{equation}\label{JKO}
\rho_\tau^n=\mbox{argmin}_{\rho\in\mP_2(\R)}\Phi_\tau\left(\rho;\rho_\tau^{n-1}\right).
\end{equation}
\end{itemize}
Once the sequence is defined we introduce the piece-wise constant interpolation as
\begin{equation}\label{eq:pic_lin}
\bar{\rho}_\tau (t)=\rho_\tau^n \qquad \mbox{for } t\in \left(\left(n-1\right)\tau,n\tau\right].
\end{equation}
\begin{lem}[{\bf Existence of minimzers}]\label{existence_Min}
Under the assumptions (g1) - (g3),(D),(gW1) and (gW2),
for any given $\rho_\tau^{n-1}\in \mP_2(\R)$ the functional $\Phi_\tau\left(\rho; \rho_\tau^{n-1}\right)$ admits a minimiser $\rho_\tau^{n}\in \mP_2(\R)$.
\end{lem}
\begin{proof}
The well-posedness of the scheme is an application of Direct Methods of Calculus of Variations. Indeed, since $\varphi$ and $V$ are non-negative then the functional $\Phi_\tau\left(\rho; \rho_\tau^{n-1}\right)$ satisfies the coercivity condition, that is, 
for any given $\eta\in \mP_2(\R)$ and for every constant $c$ we have that 
$$\inf_{\rho\in \mP_2(\R)}\{c\, W_2^2\left(\rho,\eta\right)+\mF^a\left[\rho\right]\}>-\infty\,.$$ 
Hence, for any given $\rho_\tau^{n-1}\in \mP_2(\R)$  there exists a bounded minimising sequence in $\mP_2(\R)$ that satisfies the integral condition for tightness and therefore, it is tight in $P_2(\R)$ (precompact with respect to the narrow convergence, see e.g. \cite[Remark 5.1.5]{AGS}). Moreover, by the superlinear growth condition at infinity of $\varphi$ and Dunford- Pettis Theorem we have that the minimising sequence is precompact also with respect to the weak-$L^1$ convergence and the weak-$L^1$ limit  $\rho_\tau^{n}\in K$.
The lower semicontinuity of $\Phi_\tau\left(\rho; \rho_\tau^{n-1}\right)$ with respect to the $L^1$-weak convergence easy follows by \cite{AGS}. Indeed, by \cite[Lemma 5.1.7 and Lemma 7.1.4]{AGS}, we have that the functionals $\rho\to \int_\R \rho(y) V(y)\,dy$ and $\rho\to W_2^2\left(\rho,\rho_\tau^{n-1}\right)$ are  lower semicontinuous with respect to the narrow convergence, respectively; therefore they are also $L^1$-weak lower semicontinuous. By classical results on the $L^1$- weak lower semicontinuity of integral functionals  with positive, convex and lower semicontinuous integrand,  we have that also $\rho\to\int_\R \varphi(a\rho)/ a\,dy $ is lower semicontinuous with respect to the $L^1$- weak convergence, which concludes the proof.
\end{proof}

\begin{lem}[{\bf Compactness and limit trajectory}]\label{convergence}
The piecevise constant interpolating sequence $\bar{\rho}_\tau$ narrow converges up to (non-relabelled) sub-sequence to a H\"older continuous limit curve $\rho :\left[0,\infty\right)\to \mP_2(\R)$.
\end{lem}
\begin{proof}
Directly from the definition of the minimising sequence we get, 
\begin{equation}
 \frac{1}{2\tau}\sum_{n=1}^N W_2^2\left(\rho_\tau^{n-1},\rho_\tau^n\right)\leq \mF^a\left[\rho_\tau^0\right]-\mF^a\left[\rho_\tau^N\right],
\end{equation}
which easily induces a \emph{monotonicity} property for the functional along the sequence,
$$
\mF^a\left[\rho_\tau^n\right]\leq\mF^a\left[\rho_\tau^0\right], \quad \forall n\geq0\,.
$$
Moreover,  since $\mF^a$ is non-negative we  have that
\begin{equation}
    \sum_{n=1}^\infty W_2^2\left(\rho_\tau^{n-1},\rho_\tau^n\right) \leq 2\tau \mF^a\left[\rho_\tau^0\right].
\end{equation}
Reasoning as in the proof of \cite[Theorem 11.1.6, Steps 1-2]{AGS} 
we get
\begin{equation}\label{eq:hc}
    W_2(\bar{\rho}_\tau(s),\bar{\rho}_\tau(t))\leq \sqrt{2\mF^a[\rho_\tau^0]}\max (\tau,|t-s|)^{\frac12}, \quad s,t\geq 0\,.
\end{equation}
By the refined version of Ascoli-Arzel\`a Theorem in \cite[Proposition 3.3.1]{AGS} we get the narrow convergence.
\end{proof}

We now show that the piece-wise constant interpolation sequence actually is strongly convergent in some $L^p$ space, where the exponent will depend only on the growth condition on $\varphi$ in assumption (D).
\begin{rem}
There exists a constant $C:= C(\rho^0, \varphi, a, V)$, such that
\begin{equation}\label{stima-m2}
 m_2\left[\bar{\rho}_\tau\right](T):=\int_{\R}|x|^2\bar{\rho}_\tau(T,y)dy\leq C(1+T) \mbox{ for all }T\geq 0\,.
 \end{equation}
indeed, using the classical inequality $y^2\leq 2z^2+2|y-z|^2$, calling $\gamma$ an optimal transport plan between ${\rho}^n_\tau$ and $\rho_0$

 \begin{align*}
     \int_{\R^2} y^2 d\gamma(y,z)= m_2\left[{\rho}^n_\tau\right]&\le 2\int_{\R^2} z^2 d\gamma(y,z)+2\int_{\R^2} |y-z|^2 d\gamma(y,z) \\
     &=2m_2\left[\rho_0\right] + W_2^2(\rho_\tau^n,\rho^{0})\le  2\, m_2\left[\rho_0\right] + \sum_{h=0}^n W_2^2(\rho_\tau^k,\rho^{k-1})\\
&\le 2\, m_2\left[\rho_0\right] + 2 \tau  \mF^a[\rho^0]\,.
 \end{align*}
\end{rem}

We are now going to apply Lemma \ref{Flow_interchange} with the entropy
\[
 H\left[\eta \right]=\int_\R\eta(y)\log\eta(y)\,dy,
\]
as auxiliary functional. It is well-known that $H$ posses the \emph{heat flow} as $0-$flow, that is, $G_H^s\rho_0$ is a solution of the heat equation
\[\eta_s=\eta_{yy}, \quad \eta(0,y)=\rho_0(y).\]

\begin{lem} 
There exists a constant  $A$ depending only on $\rho_0$ such that the piece-wise interpolants  $\bar{\rho}_{\tau}$  satisfy 
\begin{equation}\label{stimaH1}
\Vert \bar{\rho}_\tau^{m/2}\Vert_{L^2(0,T; H^1(\R))}\le A(1 +T),
\end{equation}
for all $T>0$. In particular, $\bar{\rho}_\tau^{m/2}\in H^1(\R)$ for every $t>0$.
\end{lem}

\begin{proof}
Let us compute
 \begin{align*}
 \frac{d}{ds}\mF^a\left[G_H^s\rho_0\right] &= \int_\R \left(\varphi^\prime(a \eta) + V(y)\right) \eta_s\,dy\\
 &=  \int_\R \varphi^\prime(a \eta)\eta_{yy}\,dy + \int_\R V(y) \eta_{yy}\,dy\\
 &=-  \int_\R \varphi^{\prime\prime}(a \eta) \partial_y(a\eta)\eta_y\,dy 
 + \int_\R V^{\prime\prime}(y) \eta\,dy\,.\\
  \end{align*}
 Thanks to assumption (D), 
  \begin{align*}
  \frac{d}{ds}\mF^a\left[G_H^s\rho_0\right] 
& \le -c_m\, \int_\R (a\eta)^{m-2} (a^\prime \eta + a \eta_y)\eta_y\,dy   + \int_\R V^{\prime\prime}(y) \eta\,dy\\
&= -c_m \int_\R\left( a^{m-2} a^\prime\eta^{m-1}+ a^{m-1}\eta^{m-2}\partial_{y}\eta\right)\partial_{y}\eta
+ \int_\R V^{\prime\prime}(y) \eta\,dy\\
&= -c_m  \frac{1}{m} \int_\R a^{m-2} a^\prime\partial_y \eta^{m}\,dy -  c_m \frac{4}{m^2} \int_\R a^{m-1}\left(\partial_{y}\eta^{m/2}\right)^2\,dy + \int_\R V^{\prime\prime}(y) \eta\,dy\\
&= c_m \frac{1}{m}\int_\R\partial_y \left(a^{m-2} a'\right) \eta^{m}\,dy-  c_m \frac{4}{m^2} \int_\R a^{m-1}\left(\partial_{y}\eta^{m/2}\right)^2\,dy + \int_\R V^{\prime\prime}(y) \eta\,dy\\
 \end{align*}
Note that by (ii) and (iii) in Proposition \ref{aProperties}
\begin{align*}\label{Stime-a}
\partial_y \left(a^{m-2} a'\right)= a^{m-1} \left((m-2)  \Bigl(\frac{a'}{a}\Bigr)^2 +  \frac{a''}{a}\right)\,\leq a^{m-1}K,
 \end{align*}
 therefore,
\begin{align*}
 \frac{d}{ds}\mF^a\left[G_H^s\rho\right] &\le 
c \int_\R  a^{m-1}  \eta^{m}(s,y)\,dy-\frac{4}{m}\int_\R \left(\partial_{y}\eta^{\frac{m}{2}}(s,y)\right)^2 \,dy
 +\int_\R V''(y)\eta(s,y)\,dy\,.\\
 \end{align*}
 We define
 $$
 \mathcal{R}[\eta] = - c \int_\R  a^{m-1}  \eta^{m}(s,y)\,dy + \frac{4}{m}\int_\R \left(\partial_{y}\eta^{\frac{m}{2}}(s,y)\right)^2 \,dy - \bar{V}
 $$
 where $\bar{V}=\sup_\R V^{\prime\prime} <+\infty$.
 By Lemma \ref{Flow_interchange}, we have that
 $$
  H\left[\rho_\tau^N \right]\leq H\left[\rho^0 \right]- \tau\sum_{n=1}^N  \mathcal{R}\left[\rho_\tau^n\right]\,;
 $$
 hence,
 \begin{eqnarray*}
 && \tau\sum_{n=1}^N \frac{4}{m}\int_\R \left(\partial_{y}(\rho_\tau^n)^{\frac{m}{2}}(s,y)\right)^2 \,dy \le H\left[\rho^0 \right] - H\left[\rho_\tau^N \right]  + c  \tau\sum_{n=1}^N  \int_\R  a^{m-1}  (\rho_\tau^n)^{m}(s,y)\,dy +  \bar{V} N\tau. 
 \end{eqnarray*}
 In particular,
 \begin{eqnarray}\label{Hstima}
 \nonumber \tau \frac{4}{m}\sum_{n=1}^N \Vert(\rho_\tau^n)^{\frac{m}{2}}\Vert^2_{H^1} &\le&
  \tau \frac{4}{m}\sum_{n=1}^N \int_\R \left(\partial_{y}(\rho_\tau^n)^{\frac{m}{2}}\right)^2 \,dy + \tau\frac{4}{m} \sum_{n=1}^N \int_\R a^{m-1}( \rho_\tau^n)^m \,dy
   \\
 &\le&
 H\left[\rho^0 \right] - H\left[\rho_\tau^N \right]  + (c+  \frac{4}{m}) \tau\sum_{n=1}^N  \int_\R  a^{m-1}  (\rho_\tau^n)^{m}\,dy +  \bar{V} N\tau 
 \end{eqnarray}
Recalling the standard inequalities
 $$
 -\frac{2}{e} s^{1/2} \le s\log s\le \frac{1}{(m-1)e} s^m
 $$
 for all $s>0$, we have that
 $$
 H(\eta)=\int_\R \eta\log \eta \,dy \le \frac{1}{(m-1)e} \int_\R \eta^m\,dy \le \frac{1}{(m-1)e} \int_\R a^{m-1} \eta^m\,dy \le
 \frac{1}{e} \mF^a (\eta)\,,
 $$
 and, on the other hand,
$$
 H(\eta)\ge -\frac{2\sqrt{\pi}}{e} \Bigl( 1 + \int_{\R}x^2 \eta\,dy\Bigr)^{1/2},
 $$
 (see e.g.  \cite{DFMatthes} Lemma 4.6). 
 By (\ref{Hstima}) we have that
 \begin{eqnarray}\label{Hstima2}
 \nonumber && \tau \frac{4}{m}\sum_{n=1}^N \Vert(\rho_\tau^n)^{\frac{m}{2}}\Vert^2_{H^1} \\
 &\le&
 \frac{1}{e} \mF^a(\rho^0) + \bar{c}  \Bigl( 1 + \int_{\R}x^2 \bar{\rho}_{\tau}(y, N\tau)\,dy\Bigr)^{1/2}
 + (c+  \frac{4}{m}) N\tau  \mF^a (\rho_0) +  \bar{V} N\tau \,.
 \end{eqnarray}
 By \eqref{stima-m2} we get the thesis.
 
\end{proof}


\begin{prop}\label{conv}
The converging sub-sequence $\bar{\rho}_\tau$ in Lemma \ref{convergence} converges to a limit $\rho$ in $L^\mu([0,T]\times \R)$ for every $T>0$, with $\mu < 3m$.
\end{prop}
\begin{proof}
We first prove the convergence in $L^m([0,T]\times \R)$. The proof is a standard application of Theorem \ref{ex_AL} and we sketch here for completeness, see also \cite[Proposition 4.8]{DFMatthes} . The  strategy is to check that the hypotehsis of Theorem \ref{ex_AL} are satisfied with   $X=L^m(\R)$ and  
\begin{equation*}
\mathcal{Y}[\rho] = \begin{cases}\int_\R \left[(\rho^{\frac{m}{2}})_y\right]^2\dy+m_2\left[\rho\right], & \rho\in \mP_2(\R),\,(\rho^{\frac{m}{2}})_y\in L^2(\R),\\
+\infty & \mbox{ otherwise,} 
\end{cases}
\end{equation*}
and
\begin{equation*}
\mathrm{d}(\rho,\eta) = \begin{cases} W_2^2(\rho,\eta) & \rho,\eta \in \mP_2(\R),\\
    +\infty & \mbox{ otherwise.} 
    \end{cases}
\end{equation*}
By Frechet-Kolmogorov Theorem (see e.g. \cite[Theorem IV.8.20]{DS88}), it can be shown that the sub-levels of $\mathcal{Y}$, $ \mathcal{Y}_c = \left\{\rho \in L^m(\R) | \mathcal{Y}[\rho]\leq c\right\}$ for  $c>0$, are relatively compact in $L^m(\R)$. The bounds in \eqref{stima-m2} and \eqref{stimaH1} imply the first condition in \eqref{eq:al}, that is
\[ 
\sup_{u\in U}\int_0^T\mathcal{Y}[u(t))] dt<\infty
\]
where $U=\left\{\bar{\rho}_{\tau_k}| k\in \mathbb{N}\right\}$. The second condition in \eqref{eq:al} is a direct consequence of the Holder continuity \eqref{eq:hc}. The hypothesis of Theorem \ref{ex_AL} are then satisfied and we can extract a sub-sequence $\bar{\rho}_{\tau_k'}$ converging in measure with respect to $t\in \left[0,T\right]$ to some limit $\rho^{*}$ in $L^m(\R)$. By Lemma \ref{convergence} $\rho^{*}$ coincides with the narrow limit $\rho$ for every $t\in \left[0,T\right]$ and so the entire sequence $\bar{\rho}_{\tau_k}$ converges in measure to $\rho$. By the uniform in $\tau$ bound on the $L^m -$norm of $\bar{\rho}_{\tau}$ and Lebesgue’s dominated convergence theorem we can argue strong convergence of $\bar{\rho}_{\tau}$ to $\rho$ in $L^m(0,T;L^m(\R))$.

Notice that, for every $T>0$ 
\begin{equation*}
    \int_0^T \|\bar{\rho}_\tau(t,\cdot)-\rho(t,\cdot)\|_{L^m(\R)}^\sigma dt \rightarrow 0,
\end{equation*}
as $\tau \to 0$, for every $\sigma>0$. Applying Gagliardo-Nirenberg inequality
$$
 \|f\|_{L^p}\leq C\|\partial_y f\|_{L^r}^\theta \|f\|_{L^q}^{1-\theta}, \quad \mbox{ whith } p=\frac{2\mu}{m},\,q=r=2,\, \theta=\frac{\mu-m}{2\mu},\,
$$
with $\mu> m$, we get
\begin{align*}
     \int_0^T\|\bar{\rho}_\tau^{\frac{m}{2}}-\rho^{\frac{m}{2}}\|_{L^p}^p dt &\leq  C\int_0^T\|\partial_y\left(\bar{\rho}_\tau^{\frac{m}{2}}-\rho^{\frac{m}{2}}\right)\|_{L^2}^{p\theta}\|\bar{\rho}_\tau^{\frac{m}{2}}-\rho^{\frac{m}{2}}\|_{L^2}^{p(1-\theta)} dt\\
     & \leq C\left(\int_0^T\|\partial_y\left(\bar{\rho}_\tau^{\frac{m}{2}}-\rho^{\frac{m}{2}}\right)\|_{L^2}^2dt\right)^{\frac{p\theta}{2}}\left(\int_0^T\|\bar{\rho}_\tau^{\frac{m}{2}}-\rho^{\frac{m}{2}}\|_{L^2}^{\gamma}dt \right)^{\frac{m-\mu\theta}{m}}.
\end{align*}
The exponent $\gamma$ is given by
$$\gamma = \frac{(1-\theta)2\mu}{m-\mu\theta},$$ 
and is a positive exponent provided $\mu < 3m$.
\end{proof}

\begin{prop}
The approximating sequence $\bar{\rho}_\tau$
converges to a weak solution $\rho$ of \eqref{eq:1} in the sense of Definition \ref{def:weak_solutions}.
\end{prop}
\begin{proof}
In order to take a lighter notation let us  denote $\rho_0$ and $\rho$ two consecutive minimisers in \eqref{JKO}. For $\ep>0$ and $\zeta\in C_c^{\infty}(\R)$, define
\[P^\ep(y)=y+\ep\zeta_y(y), \qquad \rho^\ep= P^\ep_{\#}\rho.\]
The minimality of $\rho$ gives
\[0\leq \frac{1}{2\tau}\left(W_2^2(\rho^\ep,\rho_0)-W_2^2(\rho,\rho_0)\right)+\mF^a\left[\rho^\ep\right]-\mF^a\left[\rho\right].\]
Let $T$ be the optimal map pushing $\rho_0$ to $\rho$, then by definition
\begin{align*}
& W_2^2(\rho,\rho_0)=\int_\R|y-T(y)|^2\rho_0(y)dy,\\
& W_2^2(\rho^\ep,\rho_0)\leq \int_\R|y-P^\ep\left( T(y)\right)|^2\rho_0(y)dy.
\end{align*}
Therefore, 
\begin{align}\label{I1}
 \nonumber   \frac{1}{2\tau}\left(W_2^2(\rho^\ep,\rho_0)-W_2^2(\rho,\rho_0)\right) & \leq \frac{1}{2\tau}\int_\R\left(|y-P^\ep\left(T(y)\right)|^2-|y- T(y)|^2\right)\rho_0(y)dy\\
  \nonumber & =\frac{1}{2\tau}\int_\R\left(|y-\left(T(y)+\ep\zeta_y(T(y))\right)|^2-|y- T(y)|^2\right)\rho_0(y)dy\\
 &= - \frac{\ep}{\tau}\int_\R\left(y-T(y)\right)\zeta_y(T(y))\rho_0(y)dy + o(\ep):=I_1.
\end{align}
The term involving the functional can be reformulated as
\begin{align}\label{I2+I3}
 \nonumber  \mF^a\left[\rho^\ep\right]-\mF^a\left[\rho\right]&=\int_\R \left(\frac{\varphi(a(y)\rho^\ep)}{a(y)}+V(y)\rho^\ep-\frac{\varphi(a(y)\rho)}{a(y)}-V(y)\rho\right)dy\\
& =I_2+I_3,
\end{align}
where $I_2$ and $I_3$ are defined by
\begin{equation}\label{I2}
  I_2=\int_\R \left(\frac{\varphi(a(y)\rho^\ep)}{a(y)}-\frac{\varphi(a(y)\rho)}{a(y)}\right)dy=\int_\R \left(\varphi\left(\frac{a(P^\ep(y))\rho}{1+\ep\zeta_{yy}(y)}\right)\frac{1+\ep\zeta_{yy}(y)}{a(P^\ep(y))}-\frac{\varphi(a(y)\rho)}{a(y)}\right)dy,
\end{equation}
and
\begin{align}\label{I3}
 & I_3=\int_\R \left(V(P^\ep(y))-V(y)\right)\rho(y)dy.
\end{align}
respectively. In order to handle the $I_2$ term we introduce the following function
\[
 B(\chi,\eta) = \frac{1}{\chi}\varphi(\chi\eta).
\]
A first order Taylor of $B$ expansion around the point $(\chi,\eta)$ with perturbation $(\chi^\ep,\eta^\ep)$ gives
\[
 B(\chi^\ep,\eta^\ep) = B(\chi,\eta) + \left(\frac{\eta}{\chi}\varphi'(\chi\eta)-\frac{1}{\chi^2}\varphi(\chi\eta)\right)(\chi^\ep-\chi)+\varphi'(\chi\eta)(\eta^\ep-\eta)+ R_\ep(\chi,\eta).
\]
Fix $(\chi,\eta)=(a(y),\rho)$, $(\chi^\ep,\eta^\ep)=(a(P^\ep(y)),\frac{\rho}{1+\ep\zeta_{yy}})$, then 
$I_2$ becomes
\begin{align*}
    & \int_\R\left[\frac{\varphi(a\rho)}{a}+\left(\frac{\rho}{a}\varphi'(a\rho)-\frac{\varphi(a\rho)}{a^2}\right)\left(a\circ P^\ep -a\right)+\varphi'(a\rho)\left(\frac{\ep\zeta_{yy}}{1+\ep\zeta_{yy}}\right)\rho+ R_\ep\right]\left(1+\ep\zeta_{yy}\right)-\frac{\varphi(a\rho)}{a}dy \\
    & =\ep\int_\R\frac{\varphi(a\rho)}{a}\zeta_{yy}+\left(\frac{\rho}{a}\varphi'(a\rho)-\frac{\varphi(a\rho)}{a^2}\right)\frac{a\circ P^\ep -a}{\ep}\left(1+\ep\zeta_{yy}\right)+\rho\varphi'(a\rho)\zeta_{yy}dy+ \int_\R R_\ep\left(1+\ep\zeta_{yy}\right)dy. 
\end{align*}
By dominated convergence theorem we can prove that the last term involving $R_\ep$ is $o(\ep)$. Indeed,
\begin{equation}\label{eq:resto}
    \frac{1}{\ep}\int_\R R_\ep\left(1+\ep\zeta_{yy}\right)dy= \frac{1}{\ep}\int_\R \left(R_\ep^1+R_\ep^2+R_\ep^3\right)\left(1+\ep\zeta_{yy}\right)dy,
\end{equation}
where
\begin{align*}
    & R_\ep^1 = \frac12\left(\frac{\tilde{\rho}^2}{\tilde{a}}\varphi''(\tilde{a}
    \tilde{\rho})- 2\frac{\tilde{\rho}\varphi'(\tilde{a}\tilde{\rho})}{\tilde{a}^2}+2\frac{\varphi(\tilde{a}\tilde{\rho})}{\tilde{a}^3}\right)\left(a\circ P^\ep - a\right)^2,\\
    &  R_\ep^2 = \frac12\tilde{a}\varphi''(\tilde{a}
    \tilde{\rho})\left(\frac{\ep\zeta_{yy}}{1+\ep\zeta_{yy}}\right)^2,\\
    & R_\ep^3 = \tilde{\rho}\varphi''(\tilde{a}
    \tilde{\rho})\left(\frac{\ep\zeta_{yy}}{1+\ep\zeta_{yy}}\right)\left(a\circ P^\ep - a\right),\\
\end{align*}
for some $\tilde{a}$ between $a$ and $a\circ P^\ep$ and $\tilde{\rho}$ between $\rho$ and  $\rho/(1+\ep\zeta_{yy})$. Thanks to the growth control (D) it is easy to see that the remainder goes to zero in view of the $L^\mu$ control of $\rho_\tau^n$.

Summing up all the contributions coming from \eqref{I1},  \eqref{I2+I3}, \eqref{I2},  \eqref{I3} dividing by $\ep$, sending $\ep \to 0$ and performing the same computation with $\ep< 0$ we have that
\begin{align*}
& \frac{1}{\tau}\int_\R\left(y-T(y)\right)\zeta_y(T(y))\rho_0(y)dy\\
&= \int_\R \frac{\varphi(a\rho)}{a}\zeta_{yy} + \left( \varphi'(a\rho)\frac{\rho}{a}\left(a'\zeta_y-a\zeta_{yy}\right)-\varphi(a\rho)\frac{a'}{a^2}\zeta_y\right)dy
+\int_\R V'(y)\rho(y)\zeta_y(y)dy\,.
\end{align*}
 
Note that, by Taylor expansion of $\zeta$ around $T$ we get that
$$
\frac{1}{\tau}\int_\R\left(y-T(y)\right)\zeta_y(T(y))\rho_0(y)dy
= \frac{1}{\tau}\int_\R \zeta(y)\left[\rho_0(y)-\rho(y)\right] dy + O(\tau)\,.
$$
We recall now that $\rho_0$ and $\rho$ are two consecutive minimisers in \eqref{JKO}, so that $\rho_0=\rho_\tau^{n}$ and $\rho=\rho_\tau^{n+1}$, so the equality above reads as
\begin{align*}
&\int_\R \zeta\left[\rho_\tau^{n}-\rho_\tau^{n+1}\right] dy + O(\tau)\\
&=\tau\int_\R V'\rho_\tau^{n+1}\zeta_y +\frac{\varphi(a\rho_\tau^{n+1})}{a}\zeta_{yy} + \left( \varphi'(a\rho_\tau^{n+1})\frac{\rho_\tau^{n+1}}{a}\left(a'\zeta_y-a\zeta_{yy}\right)-\varphi(a\rho_\tau^{n+1})\frac{a'}{a^2}\zeta_y\right)dy.
\end{align*}
Let $0\leq t < s$ be fixed, with
$$ h=\left[\frac{t}{\tau}\right]+1 \,\mbox{ and }\,k=\left[\frac{s}{\tau}\right].$$
Summing the equality above form $h$ to $k$ we get,
\begin{align*}
&\int_\R \zeta\left[\rho_\tau^{h}-\rho_\tau^{k+1}\right] dy + O(\tau)\\
&=\tau\sum_{n=h}^{k}\int_\R V'\rho_\tau^{n+1}\zeta_y +\frac{\varphi(a\rho_\tau^{n+1})}{a}\zeta_{yy} + \left( \varphi'(a\rho_\tau^{n+1})\frac{\rho_\tau^{n+1}}{a}\left(a'\zeta_y-a\zeta_{yy}\right)-\varphi(a\rho_\tau^{n+1})\frac{a'}{a^2}\zeta_y\right)dy,
\end{align*}
that is equivalent to 
\begin{align*}
&\int_\R \zeta\left[\bar{\rho}_\tau(t)-\bar{\rho}_\tau(s)\right] dy + O(\tau)\\
&=\int_{t}^{s}\int_\R V'\bar{\rho}_\tau(\sigma)\zeta_y +\frac{\varphi(a\bar{\rho}_\tau(\sigma))}{a}\zeta_{yy} + \left( \varphi'(a\bar{\rho}_\tau(\sigma))\frac{\bar{\rho}_\tau(\sigma)}{a}\left(a'\zeta_y-a\zeta_{yy}\right)-\varphi(a\bar{\rho}_\tau(\sigma))\frac{a'}{a^2}\zeta_y\right)dy\,d\sigma,
\end{align*}
where $\bar{\rho}_\tau$ is the piece-wise constant interpolation introduced in \eqref{eq:pic_lin}. Thanks to Lemma \ref{convergence} and Proposition \ref{conv}, together with the growth condition (D), sending $\tau \to 0$ we obtain
\begin{align*}
&\int_\R \zeta\left[\rho(t)-\rho(s)\right] dy\\
&=\int_{t}^{s}\int_\R V'\rho(\sigma)\zeta_y +\frac{\varphi(a\rho(\sigma))}{a}\zeta_{yy} + \left( \varphi'(a\rho(\sigma))\frac{\rho(\sigma)}{a}\left(a'\zeta_y-a\zeta_{yy}\right)-\varphi(a\rho(\sigma))\frac{a'}{a^2}\zeta_y\right)dy\,d\sigma.
\end{align*}
Integrating by parts in the second and in the forth  term on the r.h.s. we get 
\begin{align*}
&\int_\R \zeta\left[\rho(t)-\rho(s)\right] =\int_{t}^{s}\int_\R V'\rho(\sigma)\zeta_y + \rho(\sigma)(\varphi'(a\rho(\sigma)))_y\zeta_y dy\,d\sigma,
\end{align*}
dividing by $s - t$ and taking the limit as $s \to t$, we get the definition of weak solution in Definition \ref{def:weak_solutions}.
\end{proof}

\subsection{$\lambda$-convexity and $k$-flow}\label{k-flow}

We want study the convexity of the functional $\mF^a$ under general assumptions as in Section \ref{Assumptions}. In Subsection \ref{SpecialCases}
we will show some explicit examples as the case of heat equation, linear Fokker-Planck equation and Porous medium with degenerate mobility.

\begin{lem}\label{lem:k-flow}
Let $F^a$ and $H$ be defined as in \eqref{eq:F&H}, and let us assume that the matrix

\begin{align*}
  \mathcal{H}(y,\eta) =
  \begin{pmatrix}
    H_{yy}(y,\eta) + V''(y) - k & H_{\eta y}(y,\eta) \\
    H_{\eta y}(y,\eta) & H_{\eta\eta}(y,\eta)
  \end{pmatrix},
\end{align*}
is positive semi-definite in $\R\times\R_{+}$, that is $$H(y,\eta)+V(y)-\frac{k}{2}y^2,$$
is jointly convex on $\R\times\R_{+}$. Then the solution to \eqref{eq:2} is a $k$-flow for the functional $ \mF^a\left[\rho\right]$. 
 \end{lem}
\begin{proof}
We adapt the regularisation procedure used in \cite{DFMatthes} to our case. We first consider the following cut-off for the function $F^a$ introduced in \eqref{eq:F&H}
\begin{equation}
 \tilde{F}^N(y,\eta)=\begin{cases}
                                          F^a(y,\eta) & \mbox{ if } |y|\leq N,\\
                                          F^a(N,\eta) & \mbox{ if } y> N,\\
                                          F^a(-N,\eta) & \mbox{ if } y< -N,
                        \end{cases}
\end{equation}
and consider $F^N$ be a $C^\infty$ mollification of $\tilde{F}^N(y,\eta)$ such that $F_y^N=0$ for $|y|\geq N+1/2$ and 
\begin{equation}\label{parab}
 \eta F^N_{\eta\eta}\geq c >0.    
\end{equation}
According to \eqref{eq:F&H} we can define $H^N$ and the functional $\mF_N^a$, by replacing $F^a$ with $F^N$. The following initial-boundary value problem
\begin{equation}\label{eq:reg}
\begin{cases}
          \partial_t \rho_N= (\rho_N (\left(F_{y\eta}^N(y,\rho_N)+[\rho_N]_y F^N_{\eta\eta}(y,\rho_N)\right)+V'))_y\\
          \partial_y \rho_N(t,N)=\partial_y \rho_N(t,-N)=0\\
          \rho_N(0,y)=   \rho_{N,0} \,.
\end{cases}
\end{equation}
is uniformly parabolic, thanks to \eqref{parab}, considering an initial datum $\rho_{N,0} $ such that
$ \rho_{N,0}$ satisfies the following inequality  
\begin{equation}\label{Ineq-InitialC}
\int_\R \frac{1}{a(y)}\varphi(a(y)\rho_{N,0}(y))\dy\le \int_\R \frac{1}{a(y)}\varphi(a(y)\rho_0(y))\dy\,,
\end{equation} 
moreover the solution $\rho_N$ is supported in $\left[-N,N\right]$ and strictly positive. Hence, can define the corresponding cumulative distribution function $R^N$ and its pseudo-inverse $Y^N$ that obeys to \[   
Y_t^N = (H_\eta^N(Y^N, Y_z^N))_z -H_Y^N (Y^N, Y_z^N)-\frac{1}{ Y_z^N} (V(Y^N))_z\,,
\]
indeed first term in the right hand side in \eqref{eq:pseudo} can be rewritten in term of $H^N$ as follows
\begin{align*}
 &  -\frac{1}{ Y_z^N} (F_\eta^N(Y^N,\frac{1}{\partial_z Y^N}))_z= - (F_\eta^N(y,\rho_N))_y\\
 & =\frac{1}{\rho_N}	(H_\eta^N(y,\frac{1}{\rho_N}))_y- H_y^N (y,\frac{1}{\rho_N})	\\
 & = Y_z^N (H_\eta^N(Y^N, Y_z^N))_y-H_y^N (Y^N, Y_z^N).
 \end{align*}
We now prove that the solution to \eqref{eq:reg} is a $k$-flow, showing that the E.V.I. \eqref{eq:EVI} is satisfied. Changing variable $\mF_N^a$ we get
\begin{align*}
 & \mF_N^a\left[\rho_N\right]=\int_{-N}^{N}F^N(y,\rho_N)\dx + \int_{-N}^{N}V(y)\rho_N(y)\dy \\
 & =\int_{0}^{1} Y_z^N(t,z)F^N\left(Y^N(t,z),\frac{1}{ Y_z^N(t,z)}\right)\dz + \int_{0}^{1}V(Y^N(t,z))\dz \\
 & =\int_{0}^{1}H^N\left(Y^N(t,z), Y_z^N(t,z)\right)\dz + \int_{0}^{1}V(Y^N(t,z))\dz.
\end{align*}
Since the Wasserstein distance can be rephrased in terms of pseudo-inverse as
 \[
 W_2^2(\rho_1,\rho_2)=\int_0^1\left(Y_1-Y_2\right)^2dz, 
 \]
 for any $\rho_1$ and $\rho_2$ in $\mP_2(\R)$, we have for fixed $\tilde{\rho}_N$
 \begin{align*}
  & \frac{1}{2}\frac{d^+}{dt}W_2^2( \rho_N(t),\tilde{\rho}_N)+\frac{k}{2}W_2^2(\rho_N(t),\tilde{\rho}_N)= \frac{1}{2}\frac{d^+}{dt}\int_0^1(Y^N-\tilde{Y}^N)^2\dz+\frac{k}{2}\int_0^1(Y^N-\tilde{Y}^N)^2\dz\\
  & =\int_0^1 Y_t^N(Y^N-\tilde{Y}^N)\dz+\frac{k}{2}\int_0^1(Y^N-\tilde{Y}^N)^2\dz\\
  & =\int_0^1( (H_\eta^N(Y^N, Y_z^N))_z -H_Y^N (Y^N, Y_z^N)-\frac{1}{ Y_z^N} (V(Y^N))_z)(Y^N-\tilde{Y}^N)\dz\\
  &+\frac{k}{2}\int_0^1(Y^N-\tilde{Y}^N)^2\dz\,.
\end{align*}  
 Integrating by parts in the first term and using the convexity
 \begin{align*}
 & \int_0^1H_\eta^N(Y^N,Y_z^N)(\tilde{Y}_z^N- Y_z^N)\dz+\int_0^1H_Y^N (Y^N, Y_z^N)(\tilde{Y}^N-Y^N)\dz\\
 &+\int_0^1V_x(Y^N)(\tilde{Y}^N-Y^N)\dz+\frac{k}{2}\int_0^1(Y^N-\tilde{Y}^N)^2\dz\\
 & \leq \mF_N^a\left[\tilde{\rho}^N\right]-\mF_N^a\left[\rho_N\right]\,.
 \end{align*}
In order to conclude the proof we need to pass to the limit $N\to\infty$ the above inequality. We first need to show that the sequence $\rho_N$ converges to a certain limit function solution to \eqref{eq:1}. Let us estimate 

\begin{eqnarray*} 
\frac{1}{m}\frac{d}{dt}\int_{-N}^N a^{m-1}(y)\rho^{m}_N(y,t)\dy &=& -\Bigl(\frac{m-1}{m}\Bigr)\int_{-N}^{N}\left[(a\rho_N)^m\right]_y\frac{1}{a}\left(F_{y\eta}^N(y,\rho_N)+[\rho_N]_y F^N_{\eta\eta}(y,\rho_N)\right)\dy\\
 &&+\frac{m-1}{m}\int_{-N}^{N}(a\rho_N)^m \Bigl[\frac{V^\prime}{a}\Bigr]_y\dy\\
  &= &-\Bigl(\frac{m-1}{m}\Bigr)\int_{-N}^{N}\left[(a\rho_N)^m\right]_y\frac{[a\rho_N]_y}{a}\varphi^{\prime\prime}(a\rho_N) \dy\\
 && +\frac{m-1}{m}\int_{-N}^{N}(a\rho_N)^m \left[\frac{V^\prime}{a}\right]_y\dy\,.\\
 &= & -(m-1) \int_{-N}^{N} \frac{(a\rho_N)^{m-1}}{a} [a\rho_N]^2_y \varphi^{\prime\prime}(a\rho_N) \dy\\
 && +\frac{m-1}{m}\int_{-N}^{N}(a\rho_N)^m \left[\frac{V^\prime}{a}\right]_y\dy\,.
 \end{eqnarray*}

By the growth condition from below in (D) and (gW2) the last equality becomes
\begin{equation}\label{Ineq-InitialC1}
\frac{d}{dt}\int_{-N}^N a^{m-1}\rho_N^{m}\dy\le -\frac{4m(m-1)}{(2m-1)^2}c_m\, \int_{-N}^N\frac{1}{a} \left(\left[(a\rho_N)^{m-\frac{1}{2}}\right]_y\right)^2\dy +L (m-1)\int_{-N}^N a^{m-1}\rho_N^{m}\dy\,.
\end{equation}
By applying the Gronwall's inequality in $(0,T)$  we deduce an $L^m$ estimates on $\rho_N$ that is
\begin{equation}\label{Stima-rhom}
     \int_{-N}^N \rho_N^{m}\dy \leq 
\int_{-N}^N a^{m-1}\rho_N^{m}\dy\le e^{(m-1)L T} \int_{-N}^N a^{m-1}\rho_{N,0}^{m} \dy \le  c\, e^{(m-1)L T}  \mF^a\left[\rho_0\right]\,.
\end{equation}
This actually induce a $L^\infty$ bound in space on both $\rho_N$ and the product $a\rho_N$. The firs estimate can be trivially deduce from \eqref{Stima-rhom}. In order to see the $L^\infty$ bound on $a\rho_N$ consider the change of variable $x=\alpha^{-1}(y)$ that maps $[-N,N]$ to $[-1+\delta_N,1-\delta_N]$ for some $\delta_N>0$. Define the scaling $v_N(x,t)=a(\alpha(x))\rho_N(\alpha(x),t)$ for $x\in[-1+\delta_N,1-\delta_N]$ and zero outside. The above change of variable in \eqref{Ineq-InitialC} reads as
\begin{align*}
 \frac{d}{dt}\int_{-1+\delta_N}^{1-\delta_N} v_N^{m}\dx  &\le L(m-1)\int_{-1+\delta_N}^{1-\delta_N} v_N^{m}\,\dx \,.
\end{align*}
that is 
\begin{equation*}
    \|v_N\|_m \leq e^\frac{L(m-1)t}{m}\|v_0\|_m \mbox{ sending $m\to\infty$ }\|v_N\|_\infty \leq e^{L t}\|v_0\|_\infty.
\end{equation*}

Integrating (\ref{Ineq-InitialC}) with respect to $t\in (0, T)$ we get
\begin{equation}\label{Controllo-norma1}
    \frac{4(m-1)}{(2m-1)^2}c\, \int_0^T\int_{-N}^N\frac{1}{a(y)} \left[\left([a(y)\rho_N(y,t)]^{m-\frac{1}{2}}\right)_y\right]^2 \dy\,dt\leq C(T,m)\mF^a\left[\rho_0\right].
\end{equation}
 Note that
  \begin{align*}
 \left[\left([a(y)\rho_N(y,t)]^{m-\frac{1}{2}}\right)_y\right]^2 &= \left( \frac{2m-1}{2}\right)^2  \left( \frac{a^\prime}{a}\right)^2 a^{2m-1} (\rho_N)^{2m-1} \\
 &\qquad + a^{2m-1} \left([(\rho_N)^{m-\frac{1}{2}}]_y\right)^2 \\
 &\qquad + (2m-1) \left(\frac{a^{\prime}}{a}\right) a^{2m-1} (\rho_N)^{m-\frac{1}{2}}\, [(\rho_N)^{m-\frac{1}{2}}]_y\,;
 \end{align*} 
hence, 
\begin{align*}
 \int_{-N}^N\frac{1}{a(y)} \left(\left[(a(y)\rho_N(y,t))^{m-\frac{1}{2}}\right]_y\right)^2 \dy
&= \left( \frac{2m-1}{2}\right)^2 \int_{-N}^N  \left( \frac{a^\prime}{a}\right)^2 a^{2m-2} (\rho_N)^{2m-1}\dy \\
 &\qquad +\int_{-N}^N a^{2m-2} \left([(\rho_N)^{m-\frac{1}{2}}]_y\right)^2 \dy\\
 &\qquad + (2m-1) \int_{-N}^N \left(\frac{a^{\prime}}{a}\right) a^{2m-2} (\rho_N)^{m-\frac{1}{2}}\, [(\rho_N)^{m-\frac{1}{2}}]_y \dy\\
 &= \left( \frac{2m-1}{2}\right)^2 \int_{-N}^N  \left( \frac{a^\prime}{a}\right)^2 a^{2m-2} (\rho_N)^{2m-1}\dy \\
 &\qquad +\int_{-N}^N a^{2m-2} \left( [(\rho_N)^{m-\frac{1}{2}}]_y\right)^2 \dy\\
 &\qquad + (\frac{2m-1}{2}) \int_{-N}^N \left(\frac{a^{\prime}}{a}\right) a^{2m-2}[(\rho_N)^{2m-1}]_y\dy \,.
 \end{align*}
Since the first term in the right-hand side inequality is positive, $a\ge 1$, and $m>1$ then we can minimise
\begin{eqnarray*}
&&\int_{-N}^N\frac{1}{a(y)} \left(\left[(a(y)\rho_N(y,t))^{m-\frac{1}{2}}\right]_y\right)^2 \dy\\
&\ge& \int_{-N}^N  \left([(\rho_N)^{m-\frac{1}{2}}]_y\right)^2 \dy- (\frac{2m-1}{2})  \int_{-N}^N \left(\frac{a^{\prime\prime}}{a} + (2m-3) \left(\frac{a'}{a}\right)^2 \right) a^{2m-2}  (\rho_N)^{2m-1} \dy\,.
 \end{eqnarray*}
Therefore, by (\ref{Controllo-norma1}) we get that
 \begin{equation*}
  \int_{-N}^N  \left( [(\rho_N)^{m-\frac{1}{2}}]_y\right)^2 \dy \leq C(t,m)\mF^a\left[\rho_0\right] +
   (\frac{2m-1}{2})  \int_{-N}^N \left(\frac{a^{\prime\prime}}{a} + (2m-3) \left(\frac{a^\prime}{a}\right)^2 \right) a^{2m-2} \rho^{2m-1}_N \dy\,.
 \end{equation*}
 
 We recall that, by Proposition \ref{aProperties}, $a^{\prime\prime}/a $ and $|a^{\prime}(y) / a(y)|$ are bounded for every $y\in\R$. Hence, if we denote $K:= \displaystyle{\sup_{\R} \left(\frac{a^{\prime\prime}}{a} + (2m-3) \left(\frac{a^{\prime}}{a}\right)^2 \right)}$
 we can conclude that
  \begin{align*}
  \int_{-N}^N  \left([(\rho_N)^{m-\frac{1}{2}}]_y\right)^2 \dy &\le C(t,m) \mF^a\left[\rho_0\right]+K \int_{-N}^N  a^{2m-2}  \rho_N^{2m-1} \dy\,.
 \end{align*}
 Thanks to the $L^\infty$ control on $a\rho_N$ we can estimate
  \begin{align*}
 \int_{-N}^N  a^{2m-2}  \rho_N^{2m-1} \dy&= \int_{-N}^N  a^{m-1}  \rho^m_N \left(a^{m-1}  \rho^{m-1}_N\right)\dy\\
 &\le C \,\mF^a\left[\rho_N\right]\le C \,\mF^a\left[\rho_0\right] \,.
  \end{align*}
Since $\rho_N>0$ the $L^2$ estimate on $[(\rho_N)^{m-\frac{1}{2}}]_y$ easily implies an $L^2$ estimate 
$[(\rho_N)^{\frac{m}{2}}]_y$ and so the uniform in $N$ control of $\rho_N^{\frac{m}{2}}$ in $L^2\left(\left[0,T\right],H^1(\R)\right)$.

We now prove an $H^{-1}$ estimate on $[\rho_N^{\frac{m}{2}}]_t$. 
Let $\theta$ be a bounded function such that $\theta_{yy} =[\rho_N^{\frac{m}{2}}]_t$. We get that
\begin{align*}
  \int_{-N}^N (\theta_y)^2\dy &= -\int_{-N}^N \theta \theta_{yy} \dy = -\frac{m}{2} \int_{-N}^{N}\theta \rho_N^{{m\over 2 }-1}[\rho_N]_t \dy  \\
  & =-\frac{m}{2}\int_{-N}^{N}\theta \rho_N^{{m\over 2 }-1} \left[\rho_N\left(F_{\eta y}^N(y,\rho_N)+F_{\eta\eta}^N(y,\rho_N)\partial_y \rho_N + V^\prime\right)\right]_y\dy\\
 & = \frac{m}{2}\int_{-N}^{N}[\theta \rho_N^{{m\over 2 }-1}]_y \rho_N\left(F_{\eta y}^N(y,\rho_N)+F_{\eta\eta}^N(y,\rho_N)\partial_y \rho_N + V'\right) \dy \\
 & = \frac{m}{2}\int_{-N}^{N}[\theta \rho_N^{{m\over 2 }-1}]_y \rho_N\left(\left(a^\prime\rho_N+a\, [\rho_N]_y\right)\varphi^{\prime\prime}(a\rho_N) + V^\prime\right) \dy \\
 &= \frac{m}{2}\int_{-N}^{N}\theta_y \rho_N^{{m\over 2 }} \left(\left(a'\rho_N+a\,[\rho_N]_y\right)\varphi^{\prime\prime}(a\rho_N) + V^{\prime}\right) \dy \\
 &+ \frac{m}{2}\int_{-N}^{N}\theta [ \rho_N^{\frac{m}{2}-1}]_y \rho_N\left(\left(a^{\prime}\rho_N+a\,[\rho_N]_y\right)\varphi^{\prime\prime}(a\rho_N) + V^\prime\right) \dy\\
 & =I_1+I_2.
 \end{align*}
Applying the weighted Cauchy inequality to $I_1$ we have that
\begin{equation}
    I_1 \leq \frac12\int_{-N}^N\partial_y \theta ^2\dy + \frac{m^2}{4}3C(V,a)^2\int_{-N}^N\rho_N^m\dy + \frac{m^2}{4}\frac{3}{2} C^2\int_{-N}^N\partial_y(\rho_N^{\frac{m}{2}})^2\dy.
\end{equation}
Concerning $I_2$ its easy to see that we can estimate
\begin{align*}
    |I_2| & \leq \|\theta\|_\infty C(m,\|a\rho_N\|_\infty)\int_{-N}^N\partial_y(\rho_N^{\frac{m}{2}})^2\dy.
\end{align*}
Hence we can conclude that 
\begin{equation}
    \frac12\int_{-N}^N\partial_y \theta ^2\dy \leq C(m,\|a\rho_N\|_\infty,V)\left(\int_{-N}^N\rho_N^m\dy + \int_{-N}^N\partial_y(\rho_N^{\frac{m}{2}})^2\dy\right).
\end{equation}
Thanks to the previous bounds on $\rho_N^{\frac{m}{2}}$ we can conclude that $\partial_t \rho_N^{\frac{m}{2}}$ is $N-$uniformly bounded in $L^2(0,T;H^{-1}(\R))$, thus invoking Aubin-Lions lemma $\rho_N^{\frac{m}{2}}$ converges to a certain limit $\eta$ in $L_{loc}^2(\R_{+}\times\R)$. The estimates above allow us to pass to limit in the weak formulation of the regularised problem in order to recover weak solutions to \eqref{eq:1}. The passage to the limit in the E.V.I. can easily be deduced by the argument in \cite{DFMatthes}.
 \end{proof}
 We are now ready to state and proof the main results of this section.
\begin{thm}\label{thm:rho}
Assume that $g:[-1,1]\to [0,1]$ is under assumptions (g1), (g2) and (g3), $\varphi:[0,+\infty)\to \R$ satisfy assumption (D) and $W:[-1,1]\to\R$ satisfy (gW1) and (gW2), together with $g$. Let $\rho_0\in\mP_2(\R)$, then there exist $\rho:[0,+\infty)\to \mP_2(\R)$ weak solutions to equation \eqref{eq:main_rho_1} in the sense of Definition \ref{def:weak_solutions}. If the additional assumption of Lemma \ref{lem:k-flow} is full-filled for some $k\in\R$, there is at most one solution to \eqref{eq:main_rho_1} with initial condition $\rho_0$.
\end{thm}
\begin{proof}
The existence is proved in Subsection \ref{sec:JKO}. Concerning the uniqueness, let $\rho$ and $\eta$ be two solutions of \eqref{eq:main_rho_1} with initial data $\rho_0$ and $\eta_0$ respectively. Under the assumption of Lemma \ref{lem:k-flow}, the E.V.I. holds for both the solutions, namely for any $\tilde{\rho}\in\mP_2(\R)$
\begin{equation}\label{eq:EVI2}
 \frac{1}{2}\frac{d}{dt}W_2^2( \rho,\tilde{\rho})+\frac{k}{2}W_2^2( \rho,\tilde{\rho})\leq \mF^a(\tilde{\rho})-\mF^a(\rho),
\end{equation}
and
\begin{equation}\label{eq:EVI3}
 \frac{1}{2}\frac{d}{dt}W_2^2( \eta,\tilde{\rho})+\frac{k}{2}W_2^2( \eta,\tilde{\rho})\leq \mF^a(\tilde{\rho})-\mF^a(\eta),
\end{equation}
are satisfied. Choosing $\tilde{\rho}=\eta$ in \eqref{eq:EVI2}, $\tilde{\rho}=\rho$ in \eqref{eq:EVI3} and summing the two inequalities , by Gronwall Lemma we get the contraction property
\begin{equation}
    W_2^2(\rho(t),\eta(t))\leq e^{-kt} W_2^2(\rho_0,\eta_0),
\end{equation}
that yields uniqueness provided $\rho_0=\eta_0$.
\end{proof}

\subsection{Existence and uniqueness result in the original variable}
We are now in the position of proving our main Theorem \ref{main_th}. Fix $T>0$ and consider the initial datum $u_0\in L^1\cap L^m(\Omega)$. Define
$$
\rho_0(y)=g(\alpha^{-1}(y))u_0(\alpha^{-1}(y)), \,y\in\R,
$$
with $\alpha$ as in \eqref{eq:alpha}. The function $\rho_0$ is an admissible initial condition for \eqref{eq:main_rho_1} in the sense of Definition \ref{def:weak_solutions} and by Theorem \ref{thm:rho} there exists a unique solution that we call $\rho$ corresponding to this initial datum, with $\rho\in L^m([0,T]\times \R)$ and $\partial_y(\rho^{\frac{m}{2}})\in L^2([0,+\infty)\times \R)$. Thanks to the uniqueness result we can define in a unique way
$$
 u(x,t)=a(\alpha(x))\rho(\alpha(x),t), \quad x\in\Omega,
$$
by the usual change of variable, we get that
$$
 \int_\Omega u^m(x,t)dx = \int_\R a^{m-1}(y)\rho^m(y,t) dy.
$$
A computation very similar to the one in proof of Lemma \ref{lem:k-flow} shows that
 \begin{align*}
\frac{1}{m} \frac{d}{dt}\int_\Omega u^{m}(x,t)\dx\  & = \frac{1}{m} \frac{d}{dt} \int_\R a^{m-1}(y)\rho^m(y,t) dy\\
& \le -\frac{4(m-1)}{(2m-1)^2}c_m\, \int_\R\frac{1}{a} \left(\left[(a\rho)^{m-\frac{1}{2}}\right]_y\right)^2\dy +L \frac{m-1}{m}\int_\R a^{m-1}\rho^{m}\dy\\
 &\le -\frac{4(m-1)}{(2m-1)^2} c_m\,\int_\Omega g^2 \left(\partial_x u^{m-\frac{1}{2}}\right)^2 \,dx +L\frac{m-1}{m}\int_\Omega u^{m}\,\dx \,.
\end{align*}
Since $u_0\in L^m(\Omega)$ we get that $u\in L^m(\Omega)$. A control on $g\partial_x u^{\frac{m}{2}}$ in $L^2$ can be easily derived from the $L^2$ bound on $\partial_y\rho ^{\frac{m}{1}}$ and assumption (gW2). Changing variable in \eqref{eq:weak_solutions} we get
\begin{align*}
  \frac{d}{dt}\int_\Omega \psi(x) u(x,t) dx & =  \frac{d}{dt}\int_\R\zeta(y) \rho(y,t) dy \\
&  = -\int_\R \partial_y\left(  \varphi'(a\rho)+V(y) \right)\partial_y\zeta(y)\rho(y,t)dy\\
&  = -\int_\Omega g^2(x)\partial_x\left(  \varphi'(u)+W(x) \right)\partial_x\psi(x)u(x,t)dx 
\end{align*}
  where $\zeta(y)=\psi(\alpha^{-1}(y))$, for $y\in\R$.
\subsection{Special cases with degenerate mobility}\label{SpecialCases}
In Section \ref{k-flow} we associated to the equation
$$
 \rho_t=(\rho\left(\varphi'(a\rho)+V\right)_y)_y
$$
the functional $\mF^a\left[\rho\right]$ as in \eqref{mFa}. By using \eqref{defY} and \eqref{Yproperty}, the functional  can be rewritten in the following form 
\begin{align}\label{Fa-psi}
 \nonumber \mF^a\left[\rho\right]
  \nonumber &= \int_{\Omega'} \frac{\varphi(a(y)\rho(y))}{a(y)}\dd y + \int_{\Omega'}V(y)\rho(y)\dd y \\
  \nonumber &= \int_0^1 \frac{\varphi(a\circ Y\,\rho\circ Y)}{a\circ Y}Y_\omega\dd\omega
  \nonumber + \int_0^1V\circ Y\,\rho\circ Y\,Y_\omega\dd\omega\\
  &= \int_0^1 \psi\Big(\frac{a\circ Y}{Y_\omega}\Big)\dd\omega + \int_0^1 V\circ Y\dd\omega =: \tilde\mF^a\left[Y\right],
\end{align}
where $\psi(s)=\varphi(s)/s$.
We recall that $\lambda$-convexity of $\mF^a$ in Wasserstein is 
equivalent
to $\lambda$-convexity of $\tilde\mF^a$ in $L^2$.
The latter is implied by convexity of $f:\Omega'\times\setR_+\to\setR$ with
\begin{align}\label{def-f}
  f(p,q) = \psi\Big(\frac{a(p)}q\Big) + V(p) - \frac\lambda2p^2\,.
\end{align}

The Hessian of $f$ is given by 
\begin{align*}
  \mathcal{H}_f (p,q)=
  \begin{pmatrix}
   (a^\prime)^2 q^{-2}\,\psi^{\prime\prime}(z)+a^{\prime\prime}q^{-1}\,\psi^{\prime}(z) + V^{\prime\prime} - \lambda\,, & -aa^{\prime}q^{-3}\,\psi^{\prime\prime}(z)-a^\prime q^{-2}\,\psi^{\prime}(z) \\ \\
    -a a^\prime q^{-3}\,\psi^{\prime\prime}(z)- a^\prime q^{-2}\,\psi^{\prime}(z)\,, & a^2q^{-4}\,\psi^{\prime\prime}(z) + 2aq^{-3}\,\psi^{\prime}(z)
  \end{pmatrix}.
\end{align*}
where $z:=a(p)/q$ and we have omitted the dependence of $a$ from the variable $p$ to not overburden the notations.
We now study the $\lambda$-convexity in  three relevant cases with $g(x)= (1-x^2)^{p/2}$.

\subsection{Heat equation}
we consider the linear heat equation with a non-trivial mobility, 
  \begin{align*}
  u_t=(g^2u_x)_x = \big(g^2u[\log u]_x\big)_x\,.
  \end{align*}
By Section \ref{sec:scaling} we get the following equation in $\rho(y,t)$
\begin{equation}\label{eq:rho}
\rho_t\ = (\rho[\log (a\rho)]_y)_y = \rho_{yy} + (\rho[\log (a)]_y)_y
\end{equation}
and the corresponding functional as in \eqref{mFa} with 
 $\varphi(\rho)=\rho\log\rho$ and $V(y)= \log a$. Therefore, by \eqref{eq:help2}, \eqref{eq:help3}, and $\psi(z)=\log z$, we have that the hessian reduces to 
\begin{align*}
  \mathcal{H}_f (p,q) =
  \begin{pmatrix}
    -2g g_{xx}-\lambda & 0 \\
    0 & q^{-2}
  \end{pmatrix}.
\end{align*}

  The equation \eqref{eq:rho} is a $\lambda$-convex gradient flow, with
  \begin{align*}
    \lambda:=\inf_{y\in\R} [\log a]_{yy} = \inf_{y\in\R} \frac{a^{\prime\prime} a- (a^\prime)^2}{a^2}\circ\alpha
    = -\sup_{x\in\Omega} (g g^{\prime\prime}),
  \end{align*}
  where we have used \eqref{eq:new_coefficient}.
  Since $g(x)=(1-x^2)^{p/2}$ for some  $p>0$ then we get that
  \begin{align*}
    -g^{\prime\prime}(x)g(x) &= p\big(x(1-x^2)^{p/2-1}\big)_x(1-x^2)^{p/2} \\
    &= p(1-x^2)^{p-2}\big(1-x^2-(p-2)x^2\big) \\
    &= p(1-x^2)^{p-2}(1-(p-1)x^2);
  \end{align*}
 which implies,
  \begin{align*}
    \lambda = p\inf_{|x|<1}(1-x^2)^{p-2}(1-(p-1)x^2). 
  \end{align*}
   for $p>2$.
  
\subsection{Linear Fokker-Plank equation} We now consider a linear Fokker-Planck equation  
with degenerate mobility $g$\ie
\begin{align*}
 u_t = \big(g^2u[\log u+W]_x\big)_x.
  \end{align*}
Recalling \eqref{eq:help2} and \eqref{eq:help3}, the hessian reduces
\begin{align*}
  \mathcal{H}_f (p,q) =
  \begin{pmatrix}
    -g g^{\prime\prime}+g^2W^{\prime\prime}+gg^{\prime} W^{\prime}-\lambda & 0 \\
    0 & q^{-2}
  \end{pmatrix}.
\end{align*}
In this case we can consider a $\lambda$ that balances the diffusive and the potential part, separately; namely, we assume that there exist $\lambda_d$ and $\lambda_W$ such that
\begin{itemize}
    \item[(i)] $- g g^{\prime\prime}\geq \lambda_d$,
    \item[(ii)] $g^2W^{\prime\prime}+g g^\prime W^\prime\geq \lambda_W$.
\end{itemize}

\subsection{Porous Medium}\label{PM}
  We consider the following porous medium equation
  \begin{align*}
  u_t = \big( g^2\,[(u^m)_x+uW^\prime] \big)_x.
  \end{align*}
  Then $\varphi(s)=s^m/(m-1)$, and accordingly
  \begin{align*}
    \psi(z) = \frac{z^{m-1}}{m-1},\quad \psi'(z) = z^{m-2},\quad \psi''(z) = (m-2)z^{m-3}.
  \end{align*}
  The component $(\mathcal{H}_f)_{11}$ is thus given by
  \begin{align}\label{component11}
  (\mathcal{H}_f)_{11}
&= \frac{1} {g^{m-1} q^{m-1}}\Bigl( (m-1)(g^\prime)^2 -g^{\prime\prime} g\Bigr) +
(g^2W^{\prime\prime}  + g g^\prime W^\prime - \lambda ) ,
  \end{align}
  and the determinant becomes
  \begin{align}\label {determinant}
    \det \mathcal{H}_f 
    &= \frac{1}{g^{2(m-1)}q^{2m}} \Bigl( (m-1)(g^\prime)^2 -m \, g^{\prime\prime} g\Bigr) + \frac{m} {g^{m-1} q^{m+1}} (g^2W^{\prime\prime}+gg^\prime W^\prime-\lambda).
  \end{align}
  
   Unfortunately, only in the linear case $m=1$ both the component $11$ and the determinant are homogeneous w.r.t.\ $q$ and $g$,
  and their respective terms can be combined to balance each other.
  As soon as $m\neq1$, all terms need to be non-negative individually.
  In the case $m=2$, $(g^{\prime})^2 -g^{\prime\prime} g = 2(1+x^2)$ that is always positive, so the first entrance in the hessian is positive as soon as 
  $$g^2W^{\prime\prime} + g g^{\prime} W^{\prime} - \lambda \geq 0. $$
  In order to preserve this condition we need to impose that
  $$\frac12(g^{\prime})^2 -\, g^{\prime\prime} g = -2g^{\frac32}\left(g^{\frac12}\right)_{xx}\geq 0,$$
  that is $g^{\frac12}$ concave, that is true only for $p<4$.
  
For general $m\neq 2$, we can rewrite
  \begin{align*}
   \frac{(m-1)}{m}(g^{\prime})^2 -g^{\prime\prime} g &= -g^{2-\frac{1}{m}} (g^{\frac{1}{m}-1} g^{\prime})_x\\
  &= -m g^{2-\frac{1}{m}} (g^{\frac{1}{m}})_{xx}\,.
   \end{align*}
   Similarly,
   we can prove that
   \begin{align*}
   (m-1)(g^{\prime})^2 -g^{\prime\prime} g &= -g^{m} (g^{1-m} g^{\prime})_x\\
  &= \frac{1}{m-2} g^{m} (g^{2-m})_{xx}\,.
   \end{align*}
   Therefore,  the formulas in \eqref{component11} and \eqref{determinant} become
    \begin{align}\label{component11-BIS}
  (\mathcal{H}_f)_{11}
&= \frac{g} { q^{m-1}} \frac{(g^{2-m})_{xx}}{m-2} +
(g^2W^{\prime\prime} + g g^{\prime}W^{\prime} - \lambda ) ,
  \end{align}
  and 
  \begin{align}\label {determinant-BIS}
    \det \mathcal{H}_f 
    &= -m^2 \frac{g^{2-\frac{1}{m}}}{g^{2(m-1)}q^{2m}} (g^{\frac{1}{m}})_{xx} + \frac{m} {g^{m-1} q^{m+1}} (g^2 W^{\prime\prime}+gg^{\prime}W^{\prime} -\lambda).
  \end{align}
Let us first compute
\begin{align*}
(g^\alpha(x))_x = -\alpha p\, x\, (1-x^2)^{\alpha p/2-1},
\end{align*}
\begin{align*}
(g^\alpha(x))_{xx}= -\alpha p\, (1-x^2)^{\alpha \frac{p}{2}-2} [ 1-(\alpha p-1)x^2]\,.
\end{align*}
The term $[ 1-(\alpha p-1)x^2]$ is always positive if $\alpha=2-m<0$; hence, $g^{2-m}$ is always convex. While, $g^{2-m}$ is concave if and only if $\alpha= 2-m >0$ and $p < 2/(2-m)$. On the other side, if $\alpha=1/m$ then, $[ 1-(\alpha p-1)x^2]$ is positive if and only if $p\le 2m$ and therefore  $g^{1/m}$ is concave. 
Note that, if $p\le 2m$ and $m < 2$ then $p$ satisfies also the condition $p< 2/(2-m)$. Hence,
summarizing,
\begin{itemize}
 \item if $m>2$ then $g^{2-m}$ is convex\,;
 \item if $m< 2$ and $p< 2/(2-m)$ then $g^{2-m}$ is concave. Note that, if $m< 2$ and $p> 2/(2-m)$ then $\alpha p >2$ therefore $[ 1-(\alpha p-1)x^2]<0$ if  $(1/(\alpha p-1))<|x|<1$ and $[ 1-(\alpha p-1)x^2]> 0$ if $|x|<1/(\alpha p-1)$;
 \item if $p\le 2m$ then $g^{1/m}$ is concave;
 \item if $m< 2$ and $p\le 2m$ then $g^{2-m}$ is concave and $g^{1/m}$ concave. Indeed, if $m< 2$ then $2m= \min\{2m, 2/(2-m)\}$; hence, 
 $p\le 2m$ implies $p< 2/(2-m)$;
 \item if $m< 2$ and $2m< p< 2/(2-m)$: $g^{1/m}$ is convex and $g^{2-m}$ is concave\,.

    \end{itemize}

We can conclude that $(\mathcal{H}_f)_{11}$ and  $ \det \mathcal{H}_f $ are both positive as soon $g^2W^{\prime\prime} + g g^{\prime} W^{\prime} - \lambda \geq 0$, $m>2$ and $p\le 2m$.

\section{Possible extensions: unbounded mobilities}

\subsection{Cauchy problems on $\R$ with unbounded mobilities}
We can consider  the following equation
\begin{equation}\label{eq:main_u_2}
  \partial_t u = (\beta(x)^2 u (\varphi'(b(x)u)+V)_x)_x,
\end{equation}
posed on $(x,t)\in Q_\R$ with $\beta:\R\rightarrow [0,+\infty)$ the inverse metric factor on the real line $\R$. We assume that $\beta$ satisfies
\begin{itemize}
  \item [(B1)] $\beta \in C^1(\R)$
  \item [(B2)] $\beta(x)>0$ for all $x\in \R$ and $\inf_{x\in \R}\beta(x)=\underline{\beta}>0$.
\end{itemize}

In \cite{Lisini}, the solution to the Cauchy problem in \eqref{eq:main_u_2} was tackled by considering a variant of the theory in \cite{AGS} in which the usual Wasserstein distance is replaced by a distance constructed in the same spirit of \cite{benamou_brenier}\ie
\[d_\beta(u_1,u_2)=\inf\left\{\int_0^1\int_\R \frac{1}{\beta(x)^2} u(x,s)w(x,s)^2 dx ds\,,\,\, u(x,0)=u_1\,,\,\,u(x,1)=u_2\,,\,\, u_s+(uw)_x = 0\right\}.\]
The results in \cite{Lisini} are valid when $b$ is smooth, uniformly bounded and uniformly positive on $\R$, and they holds in arbitrary space dimension. The goal would be to use a  scaling approach similar to the one introduced in Section \ref{sec:scaling} in order to reduce \eqref{eq:main_u_2} to an equation with homogeneous mobility.

\vskip 1cm

\noindent\textsc{Acknowledgments.} We are deeply grateful to Proff. Marco Di Francesco, Daniel Matthes and Johannes Zimmer  for introducing us to the problem and for  very illuminating discussions. SF acknowledges support from the EU-funded Erasmus Mundus programme `MathMods - Mathematical models in engineering: theory, methods, and applications' at the University of L'Aquila, and from the local fund of the University
of L’Aquila `DP-LAND (Deterministic Particles for Local And Nonlocal Dynamics).

\bibliography{degenerate_mobility}

\bibliographystyle{plain}

\end{document}